%% file: RC_to_MLT.tex
\documentclass[reqno]{amsart}
\usepackage{amssymb,setspace,tikz,xcolor,mathrsfs,listings,multicol,mathabx,bm}
\usepackage{rotating}
\usepackage[vcentermath]{youngtab}
\usepackage{hyperref}
\usetikzlibrary{arrows,matrix}
\usepackage[margin=1.25in]{geometry}
\tikzset{tab/.style={matrix of math nodes,column sep=-.35, row sep=-.35,text height=7pt,text width=7pt,align=center,inner sep=2,font=\footnotesize}}

\newcommand{\bon}{\overline{1}}
\newcommand{\btw}{\overline{2}}
\newcommand{\bth}{\overline{3}}

\newcommand{\bi}{\overline\imath}

\newcommand{\brn}{\overline{n}}

\newcommand{\g}{\mathfrak{g}}
\newcommand{\HH}{\mathcal{H}}
\newcommand{\id}{\operatorname{id}}
\newcommand{\inner}[2]{\left\langle #1, #2 \right\rangle}
\newcommand{\iso}{\cong}
\newcommand{\longdoublearrow}{\relbar\joinrel\twoheadrightarrow}
\newcommand{\longhookarrow}{\lhook\joinrel\longrightarrow}
\renewcommand{\mid}{:}

\newcommand{\RC}{\operatorname{RC}} 

\newcommand{\TT}{\mathcal{T}} 
\newcommand{\seg}{\operatorname{seg}} 
\newcommand{\rem}{\operatorname{rem}} 
\newcommand{\rpt}{\operatorname{rpt}} 
\newcommand{\diff}{\operatorname{diff}} 
\newcommand{\wt}{\mathrm{wt}}
\newcommand{\ZZ}{\mathbf{Z}}

\newcommand{\absval}[1]{\left\lvert #1 \right\rvert}

\usepackage{listings}
\lstdefinelanguage{Sage}[]{Python}
{morekeywords={False,sage,True},sensitive=true}
\definecolor{dblackcolor}{rgb}{0.0,0.0,0.0}
\definecolor{dbluecolor}{rgb}{0.01,0.02,0.7}
\definecolor{dgreencolor}{rgb}{0.2,0.4,0.0}
\definecolor{dgraycolor}{rgb}{0.30,0.3,0.30}

\usepackage{xparse}

\makeatletter
\protected\def\specialmergetwolists{%
  \begingroup
  \@ifstar{\def\cnta{1}\@specialmergetwolists}
    {\def\cnta{0}\@specialmergetwolists}%
}
\def\@specialmergetwolists#1#2#3#4{%
  \def\tempa##1##2{%
    \edef##2{%
      \ifnum\cnta=\@ne\else\expandafter\@firstoftwo\fi
      \unexpanded\expandafter{##1}%
    }%
  }%
  \tempa{#2}\tempb\tempa{#3}\tempa
  \def\cnta{0}\def#4{}%
  \foreach \x in \tempb{%
    \xdef\cnta{\the\numexpr\cnta+1}%
    \gdef\cntb{0}%
    \foreach \y in \tempa{%
      \xdef\cntb{\the\numexpr\cntb+1}%
      \ifnum\cntb=\cnta\relax
        \xdef#4{#4\ifx#4\empty\else,\fi\x#1\y}%
        \breakforeach
      \fi
    }%
  }%
  \endgroup
}
\makeatother

\usepackage{xparse}
\DeclareDocumentCommand\rpp{ m m g }{
	\foreach \x [count=\s from 1] in {#1}{
	        {\ifnum\s=1
	                \draw (0,-\s)--(\x,-\s);
	                \fi}
	   \draw (0,-\s-1) to (\x,-\s-1);
	   \foreach \y in {0, ..., \x} {\draw (\y,-\s)--(\y,-\s-1);}
	}
	\specialmergetwolists{/}{#1}{#2}\ziplist
	\foreach \x/\y [count=\yi from 1] in \ziplist{
	    \node[anchor=west,font=\scriptsize] at (\x,-\yi - .5) {$\y$};
	}
	\IfValueT {#3}
	{\foreach \z [count=\zi from 1] in {#3} {\node[anchor=east,font=\scriptsize] at (0,-\zi - .5) {$\z$};}}
	{}
}

\theoremstyle{plain}
\newtheorem{thm}{Theorem}[section]
\newtheorem{lemma}[thm]{Lemma}

\newtheorem{prop}[thm]{Proposition}
\newtheorem{cor}[thm]{Corollary}
\theoremstyle{definition}
\newtheorem{dfn}[thm]{Definition}
\newtheorem{ex}[thm]{Example}
\newtheorem{remark}[thm]{Remark}
\numberwithin{equation}{section}
\numberwithin{figure}{section}
\numberwithin{table}{section}

\begin{document}
\title[Marginally large tableaux and rigged configurations]{Connecting marginally large tableaux and rigged configurations via crystals}
\author{Ben Salisbury}
\address{Department of Mathematics, Central Michigan University, Mt. Pleasant, MI 48859}
\email{ben.salisbury@cmich.edu}
\urladdr{http://people.cst.cmich.edu/salis1bt/}
\author{Travis Scrimshaw}
\address{Department of Mathematics, University of California, Davis, CA 95616}
\email{tscrim@ucdavis.edu}
\urladdr{http://www.math.ucdavis.edu/~scrimsha/}
\thanks{T.S.\ was partially supported by NSF grant OCI-1147247.}

\begin{abstract}
We show that the bijection from rigged configurations to tensor products of Kirillov-Reshetikhin crystals extends to a crystal isomorphism between the $B(\infty)$ models given by rigged configurations and marginally large tableaux.
\end{abstract}

\maketitle

\section{Introduction}

In~\cite{KKR86,KR86}, Kerov, Kirillov, and Reshetikhin described a recursive bijection between classically highest-weight rigged configurations in type $A_n^{(1)}$ and standard Young tableaux, showing the Kostka polynomial can be expressed as a fermionic formula. This was then extended to Littlewood-Richardson tableaux and classically highest weight elements in a tensor product of Kirillov-Reshetikhin (KR) crystals in~\cite{KSS1999} for, again, type $A_n^{(1)}$. A similar bijection $\Phi$ between rigged configurations and tensor products of the KR crystal $B^{1,1}$ corresponding to the vector representation was extended to all non-exceptional affine types in~\cite{OSS03}, type $E_6^{(1)}$ in~\cite{OS12}, and $D_4^{(3)}$ in~\cite{Scrimshaw15}.

Following~\cite{KSS1999}, it was conjectured that the bijection $\Phi$ can be further extended to a tensor product of general KR crystals with the major step being the algorithm for $B^{1,1}$. This has been proven in a variety of cases~\cite{OSS03III,OSS13,S05,SS2006,SS15,Scrimshaw15}. Despite this bijection's recursive definition, it is conjectured (see for instance~\cite{SS15}) that $\Phi$ sends a combinatorial statistic called cocharge~\cite{OSS03} to the algebraic statistic called energy~\cite{HKOTY99}, proving the so-called $X = M$ conjecture of~\cite{HKOTT02,HKOTY99}. Additionally, the bijection $\Phi$ is conjectured to translate the combinatorial $R$-matrix~\cite{KKMMNN91} into the identity on rigged configurations.

The description of $\Phi$ on classically highest-weight elements led to a description of classical crystal operators in simply-laced types in~\cite{S06} and non-simply-laced finite types in~\cite{SS15}. It was shown for type $A_n^{(1)}$ in~\cite{DS06} and $D_n^{(1)}$ in~\cite{Sakamoto13} that $\Phi$ is a classical crystal isomorphism. Using virtual crystals~\cite{OSS03II}, it can be shown $\Phi$ is a classical crystal isomorphism in non-exceptional affine types~\cite{SS15}.

Rigged configurations were also extended beyond the context of highest weight classical crystals to $B(\infty)$ in~\cite{SS2015}. There is also a similar extension of the Kashiwara-Nakashima tableaux~\cite{KN94} and Kang-Misra tableaux~\cite{KM94} (for type $G_2$), which are used to describe KR crystals~\cite{FOS09,KMOY07,Yamane98}, to the marginally large tableaux model of $B(\infty)$~\cite{HL08,HL12}. The goal of this paper is to connect with a crystal isomorphism these two models by using the bijection $\Phi$. In particular, the crystal isomorphism we obtain from Corollary~\ref{cor:bijection} is given combinatorially, in the sense that the description does not use the Kashiwara operators.

This paper is organized as follows. In Section~\ref{sec:crystals}, we give a background on crystals. In Section~\ref{sec:tableaux}, we describe the tableaux model for highest weight crystals and marginally large tableaux. In Section~\ref{sec:rigcon}, we give background on rigged configurations and describe the bijection $\Phi$. In Section~\ref{sec:isomorphism}, we construct our isomorphism between the rigged configuration model and marginally large tableaux model for $B(\infty)$. In Section~\ref{sec:statistics}, we describe certain statistics on highest weight crystals and $B(\infty)$.

\section{Crystals}
\label{sec:crystals}

Let $\g$ be a finite-dimensional simple Lie algebra with index set $I$, Cartan matrix $A = (A_{ab})_{a,b\in I}$, weight lattice $P$, root lattice $Q$, fundamental weights $\{\Lambda_a \mid a \in I\}$, simple roots $\{\alpha_a \mid a\in I\}$, and simple coroots $\{h_a \mid a\in I\}$.  There is a canonical pairing $\langle\ ,\ \rangle\colon P^\vee \times P \longrightarrow \ZZ$ defined by $\langle h_a, \alpha_b \rangle = A_{ab}$, where $P^{\vee} = \bigoplus_{a\in I} \ZZ h_a$ is the dual weight lattice.  The quantum group associated to $\g$ is denoted $U_q(\g)$, though we will not need the details concerning $U_q(\g)$.  The interested reader is encouraged to see \cite{HK02,Lusztig93} for more details.

An \emph{abstract $U_q(\g)$-crystal} is a nonempty set $B$ together with maps
\[
e_a, f_a \colon B \longrightarrow B\sqcup\{0\}, \ \ \
\varepsilon_a, \varphi_a\colon B \longrightarrow \ZZ\sqcup\{-\infty\}, \ \ \
\wt \colon B \longrightarrow P,
\]
subject to the conditions
\begin{enumerate}
\item $f_a b = b'$ if and only if $b = e_a b'$ for $b,b'\in B$ and $a \in I$;
\item if $f_a b \neq 0$, then $\wt(f_a b) = \wt(b) - \alpha_a$ for all $a\in I$; and
\item $\varphi_a(b) - \varepsilon_a(b) = \langle h_a, \wt(b) \rangle$ for all $b\in B$ and $a \in I$.
\end{enumerate}
The maps $\{e_a \mid a \in I\}$ are called the \emph{Kashiwara raising operators} and the maps $\{f_a \mid a\in I\}$ are called the \emph{Kashiwara lowering operators}.

\begin{ex}
For a dominant integral weight $\lambda$, the crystal basis 
\[
B(\lambda) = \{ f_{a_k}\cdots f_{a_1}u_{\lambda} \mid a_1,\ldots,a_k \in I,\ k\in\ZZ_{\ge0} \} \setminus \{0\}
\] 
of an irreducible, highest weight $U_q(\g)$-module $V(\lambda)$ is an abstract $U_q(\g)$-crystal. (See~\cite{HK02,K91} for details.)  The crystal $B(\lambda)$ is characterized by the following properties.
\begin{enumerate}
\item The element $u_\lambda \in B(\lambda)$ is the unique element such that $\wt(u_\lambda) = \lambda$.
\item For all $a \in I$, $e_a u_{\lambda} = 0$.
\item For all $a \in I$, $f_a^{\inner{h_a}{\lambda} + 1} u_{\lambda} = 0$. 
\end{enumerate}
\end{ex}

\begin{ex}
The crystal basis
\[
B(\infty) = \{f_{a_k}\cdots f_{a_1} u_{\infty} \mid a_1,\dots,a_k\in I,\ k\in\ZZ_{\ge0}\}
\] 
of the negative half $U_q^-(\g)$ of the quantum group (equivalently a Verma module of highest weight $0$) is a $U_q(\g)$-crystal.  (See~\cite{HK02,K91} for details.)  Some important properties of $B(\infty)$ are the following.
\begin{enumerate}
\item The element $u_\infty \in B(\infty)$ is the unique element such that $\wt(u_\infty) = 0$.
\item For all $a\in I$, $e_au_\infty =0$.
\item For any sequence $(a_1,\dots,a_k)$ from $I$, $f_{a_k}\cdots f_{a_1}u_\infty \neq 0$.
\end{enumerate}
\end{ex}


Let $B_1$ and $B_2$ be two abstract $U_q(\g)$-crystals.  A \emph{crystal morphism} $\psi\colon B_1 \longrightarrow B_2$ is a map $B_1\sqcup\{0\} \longrightarrow B_2 \sqcup \{0\}$ such that
\begin{enumerate}
\item $\psi(0) = 0$;
\item if $b \in B_1$ and $\psi(b) \in B_2$, then $\wt(\psi(b)) = \wt(b)$, $\varepsilon_a(\psi(b)) = \varepsilon_a(b)$, and $\varphi_a\psi(b)) = \varphi_a(b)$;
\item for $b \in B_1$, we have $\psi(e_a b) = e_a \psi(b)$ provided $\psi(e_ab) \neq 0$ and $e_a\psi(b) \neq 0$;
\item for $b\in B_1$, we have $\psi(f_a b) = f_a \psi(b)$ provided $\psi(f_ab) \neq 0$ and $f_i\psi(b) \neq 0$.
\end{enumerate}
A morphism $\psi$ is called \emph{strict} if $\psi$ commutes with $e_a$ and $f_a$ for all $a \in I$.  Moreover, a morphism $\psi\colon B_1 \longrightarrow B_2$ is called an \emph{embedding} if the induced map $B_1 \sqcup\{0\} \longrightarrow B_2 \sqcup \{0\}$ is injective.

Again let $B_1$ and $B_2$ be abstract $U_q(\g)$-crystals.  The tensor product $B_2 \otimes B_1$ is defined to be the Cartesian product $B_2\times B_1$ equipped with crystal operations defined by
\begin{align*}
e_a(b_2 \otimes b_1) &= \begin{cases}
e_a b_2 \otimes b_1 & \text{if } \varepsilon_a(b_2) > \varphi_a(b_1) \\
b_2 \otimes e_a b_1 & \text{if } \varepsilon_a(b_2) \le \varphi_a(b_1),
\end{cases} \\
f_a(b_2 \otimes b_1) &= \begin{cases}
f_a b_2 \otimes b_1 & \text{if } \varepsilon_a(b_2) \ge \varphi_a(b_1) \\
b_2 \otimes f_a b_1 & \text{if } \varepsilon_a(b_2) < \varphi_a(b_1),
\end{cases} \\ 
\varepsilon_a(b_2 \otimes b_1) &= \max\big( \varepsilon_a(b_2), \varepsilon_a(b_1) - \inner{h_a}{\wt(b_2)} \bigr) \\
\varphi_a(b_2 \otimes b_1) &= \max\big( \varphi_a(b_1), \varphi_a(b_2) + \inner{h_a}{\wt(b_1)} \bigr) \\
\wt(b_2 \otimes b_1) &= \wt(b_2) + \wt(b_1).
\end{align*}

\begin{remark}
Our convention for tensor products is opposite the convention given by Kashiwara in~\cite{K91}.
\end{remark}


We say an abstract $U_q(\g)$-crystal is simply a \emph{$U_q(\g)$-crystal} if it is crystal isomorphic to the crystal basis of an integrable $U_q(\g)$-module.

\section{Tableaux model}
\label{sec:tableaux}

Let $\g$ be of finite classical type. We review the Kashiwara-Nakashima tableaux and Kang-Misra tableaux, which we call \emph{classical tableaux}, model for highest weight crystals $B(\lambda)$ and the marginally large tableaux model for $B(\infty)$.

\subsection{Fundamental crystals and classical tableaux}

Recall that a tableau is called semistandard over an alphabet $J = \{ j_1 \prec j_2 \prec \cdots \prec j_p \}$ if entries are weakly increasing in rows and strictly increasing in columns, with respect to $\prec$.  Let $J(X_n)$ be the alphabet for the semistandard tableaux of type $X_n$.  Then
\begin{equation}\label{eq:alphabets}
\begin{aligned}
J(A_n) &= \{ 1 \prec 2 \prec \cdots \prec n+1 \} ,\\
J(B_n) &= \{ 1 \prec \cdots \prec n \prec 0 \prec \overline n \prec \cdots \prec \overline1 \},\\
J(C_n) &= \{ 1 \prec \cdots \prec n \prec \overline n \prec \cdots \prec \overline1 \},\\
J(D_{n+1}) &= \{ 1 \prec \cdots \prec n \prec n+1,\overline{n+1}, \prec \overline{n} \prec \cdots \prec \overline1 \}, \\
J(G_2) &= \{ 1 \prec 2 \prec 3 \prec 0 \prec \overline{3} \prec \overline{2} \prec \overline{1} \}.
\end{aligned}
\end{equation}

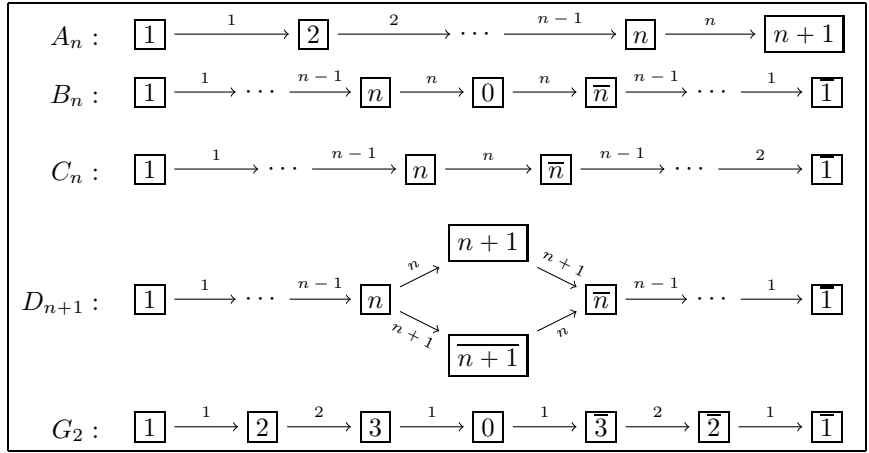
\begin{figure}[t]
\[
\begin{array}{|rl|}\hline
A_n: & 
\begin{tikzpicture}[scale=1.45,baseline=-4]
\node (1) at (0,0) {$\young(1)$};
\node (2) at (1.5,0) {$\young(2)$};
\node (d) at (3.0,0) {$\cdots$};
\node (n-1) at (4.5,0) {$\young(n)$};
\node (n) at (6,0) {$\boxed{n+1}$};
\draw[->] (1) to node[above]{\tiny$1$} (2);
\draw[->] (2) to node[above]{\tiny$2$} (d);
\draw[->] (d) to node[above]{\tiny$n-1$} (n-1);
\draw[->] (n-1) to node[above]{\tiny$n$} (n);
\end{tikzpicture}\\
B_n: &
\begin{tikzpicture}[baseline=-4]
\node (1) at (0,0) {$\young(1)$};
\node (d1) at (1.5,0) {$\cdots$};
\node (n) at (3,0) {$\young(n)$};
\node (0) at (4.5,0) {$\young(0)$};
\node (bn) at (6,0) {$\young(\brn)$};
\node (d2) at (7.5,0) {$\cdots$};
\node (b1) at (9,0) {$\young(\bon)$};
\draw[->] (1) to node[above]{\tiny$1$} (d1);
\draw[->] (d1) to node[above]{\tiny$n-1$} (n);
\draw[->] (n) to node[above]{\tiny$n$} (0);
\draw[->] (0) to node[above]{\tiny$n$} (bn);
\draw[->] (bn) to node[above]{\tiny$n-1$} (d2);
\draw[->] (d2) to node[above]{\tiny$1$} (b1);
\end{tikzpicture} 
\\[10pt]
C_n: &
\begin{tikzpicture}[baseline=-4]
\node (1) at (0,0) {$\young(1)$};
\node (d1) at (1.8,0) {$\cdots$};
\node (n) at (3.6,0) {$\young(n)$};
\node (bn) at (5.4,0) {$\young(\brn)$};
\node (d2) at (7.2,0) {$\cdots$};
\node (b1) at (9,0) {$\young(\bon)$};
\draw[->] (1) to node[above]{\tiny$1$} (d1);
\draw[->] (d1) to node[above]{\tiny$n-1$} (n);
\draw[->] (n) to node[above]{\tiny$n$} (bn);
\draw[->] (bn) to node[above]{\tiny$n-1$} (d2);
\draw[->] (d2) to node[above]{\tiny$2$} (b1);
\end{tikzpicture}
\\[10pt]
D_{n+1}: & 
\begin{tikzpicture}[baseline=-4]
\node (1) at (0,0) {$\young(1)$};
\node (d1) at (1.5,0) {$\cdots$};
\node (n-1) at (3,0) {$\young(n)$};
\node (n) at (4.5,.75) {$\boxed{n+1}$};
\node (bn) at (4.5,-.75) {$\boxed{\overline{n+1}}$};
\node (bn-1) at (6,0) {$\young(\brn)$};
\node (d2) at (7.5,0) {$\cdots$};
\node (b1) at (9,0) {$\young(\bon)$};
\draw[->] (1) to node[above]{\tiny$1$} (d1);
\draw[->] (d1) to node[above]{\tiny$n-1$} (n-1);
\draw[->] (n-1) to node[above,sloped]{\tiny$n$} (n);
\draw[->] (n-1) to node[below,sloped]{\tiny$n+1$} (bn);
\draw[->] (n) to node[above,sloped]{\tiny$n+1$} (bn-1);
\draw[->] (bn) to node[below,sloped]{\tiny$n$} (bn-1);
\draw[->] (bn-1) to node[above]{\tiny$n-1$} (d2);
\draw[->] (d2) to node[above]{\tiny$1$} (b1);
\end{tikzpicture}
\\[30pt]
G_2: & 
\begin{tikzpicture}[baseline=-4]
\node (1) at (0,0) {$\young(1)$};
\node (2) at (1.5,0) {$\young(2)$};
\node (3) at (3,0) {$\young(3)$};
\node (0) at (4.5,0) {$\young(0)$};
\node (b3) at (6,0) {$\young(\bth)$};
\node (b2) at (7.5,0) {$\young(\btw)$};
\node (b1) at (9,0) {$\young(\bon)$};
\path[->,font=\tiny]
 (1) edge node[above]{$1$} (2)
 (2) edge node[above]{$2$} (3)
 (3) edge node[above]{$1$} (0)
 (0) edge node[above]{$1$} (b3)
 (b3) edge node[above]{$2$} (b2)
 (b2) edge node[above]{$1$} (b1);
\end{tikzpicture}
\\\hline
\end{array}
\]
\caption{The fundamental crystals $\TT(\Lambda_1)$ when the underlying Lie algebra is of finite type.}\label{fig:fundcrystal}
\end{figure}

For our purposes, we need only define highest weight crystals for specific fundamental weights.  Namely, define a subset $\widehat P^+$ of $P^+$ by
\[
\widehat P^+ = \begin{cases}
 \ZZ\Lambda_1 \oplus \cdots \oplus \ZZ\Lambda_{n-1} \oplus \ZZ\Lambda_n & \text{ if } \g = A_n, C_n, \\
 \ZZ\Lambda_1 \oplus \cdots \oplus \ZZ\Lambda_{n-1} \oplus \ZZ(2\Lambda_n) & \text{ if } \g = B_n, \\
 \ZZ\Lambda_1 \oplus \cdots \oplus \ZZ\Lambda_{n-1} \oplus \ZZ(\Lambda_n+\Lambda_{n+1}) & \text{ if } \g = D_{n+1}, \\
 \ZZ\Lambda_1 \oplus \ZZ\Lambda_2 & \text{ if } \g = G_2.
\end{cases}
\]
The reason $\widehat P^+$ suffices is due to the constructions we will use in what follows.  In particular, these weights suffice to define the marginally large tableaux model $\TT(\infty)$ of the crystal $B(\infty)$, and thus will be sufficient for us to define our crystal morphism from the rigged configuration model for $B(\infty)$ to $\TT(\infty)$.  These weights also ensure that we will not have any ``spin columns'' in types $B_n$ and $D_{n+1}$.

Recall the fundamental crystals $\TT(\Lambda_1)$ in Figure~\ref{fig:fundcrystal}. Next consider some $\lambda \in \widehat P^+$; we wish to model an element in $B(\lambda)$. It is from these fundamental crystals that the more general crystals will be defined.  For $\lambda \in \widehat P^+$ defined by
\[
\lambda = \begin{cases}
 c_1\Lambda_1 + \cdots + c_{n-1}\Lambda_{n-1} + c_n\Lambda_n & \text{ if } \g = A_n, C_n, \\
 c_1\Lambda_1 + \cdots + c_{n-1}\Lambda_{n-1} + c_n(2\Lambda_n) & \text{ if } \g = B_n, \\
 c_1\Lambda_1 + \cdots + c_{n-1}\Lambda_{n-1} + c_n(\Lambda_n+\Lambda_{n+1}) & \text{ if } \g = D_{n+1}, \\
 c_1\Lambda_1 + c_2\Lambda_2  & \text{ if } \g = G_2,
\end{cases}
\]
let $Y_\lambda$ be the Young diagram with $c_i$ columns of height $i$.  Define $T_\lambda$ to be the unique tableau of shape $Y_\lambda$ such that all entries in the $j$th row of $T_\lambda$ are $j$.   We may embed $T_\lambda$ into $\TT(\Lambda_1)^{\otimes|\lambda|}$, where $|\lambda|$ is the number of boxes in $Y_\lambda$, by reading the tableaux entries from top-to-bottom starting with the right-most column.  Then $f_aT_\lambda$, for $a\in I$, is defined using the tensor product rule and the corresponding fundamental crystal above.  Now let $\TT(\lambda)$ be the set generated by $f_a$ $(a \in I)$ and $T_\lambda$.  This is the set of \emph{classical tableaux} of shape $\lambda$.  The description of type $A_n, B_n, C_n, D_n$ tableaux is due to Kashiwara and Nakashima~\cite{KN94} and type $G_2$ tableaux is due to Kang and Misra \cite{KM94}.
The resulting set is a crystal of semistandard tableaux (with respect to $J(X_n)$) satisfying certain filling conditions.  The explicit description of these crystals may be found in \cite{HK02,KM94,KN94}.

\begin{figure}[ht]
\input{./B2-21.tex}
\caption{The crystal graph $\TT(\Lambda_1+2\Lambda_2)$ of type $B_2$, created using Sage.}
\end{figure}

\subsection{Marginally large tableaux}

Following~\cite{Cliff98}, a semistandard tableau is called \emph{large} if the difference of the number of boxes in the $i$-th row containing the element $i$ and the total number of boxes in the $(i+1)$-th row is positive.  Additionally, following \cite{HL08}, a large (semistandard) tableau is called \emph{marginally large} if the aforementioned difference exactly 1.  Such tableaux are defined for simple Lie algebras $\g$ of type $A_n$, $B_n$, $C_n$, $D_{n+1}$, and $G_2$ in~\cite{HL08}, and of type $E_6$, $E_7$, $E_8$, and $F_4$ in~\cite{HL12}.  The alphabets over which each tableaux from \cite{HL08} are defined are given in Equation \eqref{eq:alphabets}.

The set of marginally large tableaux may be generated through successive application of the Kashiwara lowering operators $f_a$ ($a \in I$) to a specified highest weight vector.  It is in this way that the set of marginally large tableaux work as a combinatorial model for $B(\infty)$.  In certain types, additional conditions are required to precisely define the model, so we give the list of conditions for each type-by-type.

\begin{dfn}\label{ABCDG-tab}
For $X_n = A_n, B_n, C_n, D_{n+1}, G_2$, define the set $\TT(\infty)$ as follows.  (By convention, we assume $n=2$ when $X_n = G_2$.)
\begin{enumerate}
\item The highest weight vector is the unique tableau which consists of $n+1-i$ $i$-colored boxes in the $i$th row from the top (when written using the English convention).
\item Each element is marginally large, semistandard with respect to $J(X_n)$, and consists of exactly $n$ rows.
\end{enumerate}
We also have the following additional type-specific requirements.
\begin{itemize}
\item $X_n = B_n$ $(n\ge 2)$
\begin{enumerate}
\item Elements in the $i$th row are $\preceq \overline\imath$.
\item A $0$-box may occur at most once in a given row.
\end{enumerate}
\item $X_n = C_n$ $(n\ge 2)$
\begin{enumerate}
\item Elements in the $i$th row are $\preceq \overline\imath$.
\end{enumerate}
\item $X_n = D_{n+1}$ $(n\ge 3)$
\begin{enumerate}
\item Elements in the $i$th row are $\preceq \overline\imath$.
\item Both $n+1$ and $\overline{n+1}$ may not appear simultaneously in a single row.
\end{enumerate}
\item $X_n = G_2$
\begin{enumerate}
\item Elements in the second row are $\preceq 3$.
\item A $0$-box may occur at most once in a given row.
\end{enumerate}
\end{itemize}
\end{dfn}

The crystal operators are defined in the same way as in $\TT(\lambda)$.  Namely, read entries of a tableau $T \in \TT(\infty)$ from top-to-bottom in columns starting with the right-most column to obtain an element of $\TT(\Lambda_1)^{\otimes N}$, where $N$ is the number of boxes in $T$.  Then apply the tensor product rule to obtain $f_aT$ and $e_aT$, $a\in I$.

\begin{thm}[{\cite{HL08}}] 
Let $\g$ be a finite simple Lie algebra of type $X_n$.  Then $\TT(\infty) \cong B(\infty)$ as $U_q(\g)$-crystals.
\end{thm}

\begin{ex}
Consider type $A_3$.  The top part of the crystal graph $\TT(\infty)$ is shown in Figure~\ref{A3mlt} down to depth 3.  The notation used at the vertices is condensed so that all place holding $i$-boxes in the $i$th row are removed.  If the resulting reduction yields a row with no boxes, then that row appears with one box containing $*$.  The graph in the figure is modified output from Sage \cite{combinat,sage}.
\begin{figure}[ht]
\input{a3mlt}
\caption{The top part of $\TT(\infty)$ for type $A_3$.}
\label{A3mlt}
\end{figure}
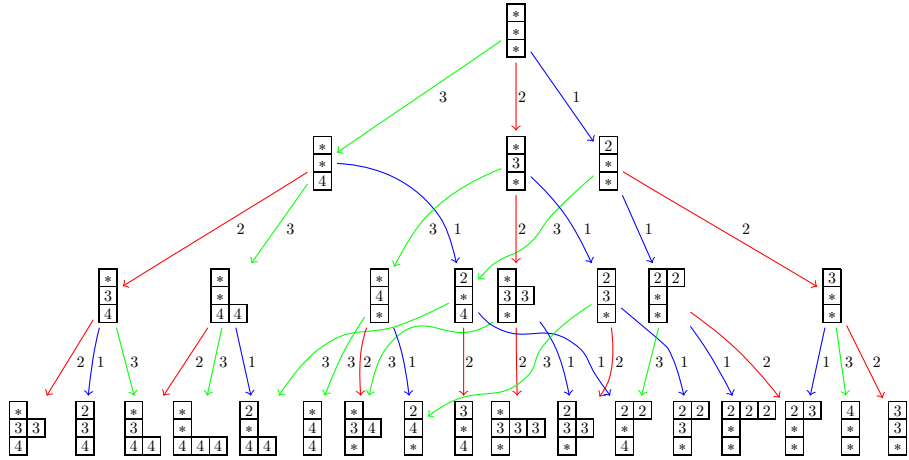
\end{ex}


Following \cite{HL12}, we call a column of any tableau $T$ a \emph{basic column} if it has height $r$ and is filled with $(1, \ldots, r)$. From~\cite{HL08}, consider the set
\[
\TT^L := \left\{T \in \bigcup_{\lambda \in \widehat P_+} \TT(\lambda) \mid T \text{ is large} \right\}
\]
We may partition $\TT^L$ into equivalence classes by saying $T \sim T'$ if they differ only by basic columns.  Note that $f_a T \neq 0$ for all $T \in \TT^L$. If $f_a T$ is large, then for all $S \in [T]$ such that $f_a S$ is large, we have $f_a S \in [f_a T]$. In other words, the crystal operators essentially preserve equivalence classes. Moreover, if $f_a S \notin [f_a T]$, then $f_a S$ differs from a unique element in $[f_a T]$ only by adding a single basic column of height $a$. Additionally, every equivalence class has exactly one marginally large tableaux. The details of these statements can be found in~\cite{HL08}.

\section{Rigged configurations}
\label{sec:rigcon}

\subsection{Crystal structure}


We first need to consider an affine type $\widetilde{\g}$ whose classical subalgebra is $\g$. However we do not do so in the usual fasion by taking the untwisted affine algebra, but instead consider those given by Table~\ref{table:affine}.
\begin{table}[ht]
\doublespacing
\[
\begin{array}{|c|c|c|c|c|c|}\hline
\g & A_n & B_n & C_n & D_{n+1} & G_2 \\\hline
\widetilde{\g} & A_n^{(1)} & D_{n+1}^{(2)} & A_{2n-1}^{(2)} & D_{n+1}^{(1)} & D_4^{(3)} \\\hline
\end{array}
\]
\singlespacing
\caption{The association of affine type $\widetilde{\g}$ with a classical type $\g$ used here.}
\label{table:affine}
\end{table}

Set $\HH = I \times \ZZ_{>0}$.  Consider a multiplicity array
\[
L = \big(L^{(a)} \in \ZZ_{\ge0} \mid a \in I \big)
\]
and a dominant integral weight $\lambda$ for $\g$.  We call a sequence of partitions $\nu = \{ \nu^{(a)} \mid a \in I \}$ an $(L, \lambda)$-\emph{configuration} if 
\begin{equation}\label{LL-config}
\sum_{(a,i)\in\HH} i m_i^{(a)} \alpha_a = \sum_{a \in I} L^{(a)} \Lambda_a - \lambda,
\end{equation}
where $m_i^{(a)}$ is the number of parts of length $i$ in the partition $\nu^{(a)}$ and $\{\alpha_a \mid a \in I\}$ are the simple roots for $\g$.  The set of all such $(L,\lambda)$-configurations is denoted $C(L,\lambda)$.
To an element $\nu \in C(L,\lambda)$, define the \emph{vacancy numbers} of $\nu$ to be 
\begin{equation}
\label{eq:vacancy}
p_i^{(a)}(\nu) = p_i^{(a)} = L^{(a)} - \sum_{(b,j) \in \HH_0} A_{ab} \min(i, j) m_j^{(b)}.
\end{equation}

Recall that a partition is a multiset of integers (typically sorted in weakly decreasing order).  More generally, a \emph{rigged partition} is a multiset of pairs of integers $(i, x)$ such that $i > 0$ (typically sorted under weakly decreasing lexicographic order).  Each $(i,x)$ is called a \emph{string}, while $i$ is called the length or size of the string and $x$ is the \emph{rigging}, \emph{label}, or \emph{quantum number} of the string.  Finally, a \emph{rigged configuration} is a pair $(\nu, J)$ where $\nu \in C(L,\lambda)$ and $J = \big( J_i^{(a)} \big)_{(a, i) \in \HH}$, where each $J_i^{(a)}$ is a weakly decreasing sequence of riggings of strings of length $i$ in $\nu^{(a)}$. We call a rigged configuration \emph{valid} if every label $x \in J_i^{(a)}$ satisfies the inequality $p_i^{(a)} \geq x$ for all $(a, i) \in \HH$. We say a rigged configuration is \emph{highest weight} if $x \geq 0$ for all labels $x$. Define the \emph{colabel} or \emph{coquantum number} of a string $(i,x)$ to be $p_i^{(a)} - x$.  
For brevity, we will often denote the $a$th part of $(\nu,J)$ by $(\nu,J)^{(a)}$ (as opposed to $(\nu^{(a)},J^{(a)})$).

\begin{dfn}
Let $(\nu_\emptyset,J_\emptyset)$ be the rigged configuration with empty partition and empty riggings and let $L$ be the multiplicity array of all zeros.
Define $\RC(\infty)$ to be the graph generated by $(\nu_\emptyset,J_\emptyset)$, $e_a$, and $f_a$, for $a \in I$, where $e_a$ and $f_a$ acts on elements $(\nu,J)$ in $\RC(\infty)$ as follows. Fix $a \in I$ and let $x$ be the smallest label of $(\nu,J)^{(a)}$.
\begin{itemize}
\item[$e_a$:] If $x \geq 0$, then set $e_a(\nu,J) = 0$. Otherwise, let $\ell$ be the minimal length of all strings in $(\nu,J)^{(a)}$ which have label $x$.  The rigged configuration $e_a(\nu,J)$ is obtained by replacing the string $(\ell, x)$ with the string $(\ell-1, x+1)$ and changing all other labels so that all colabels remain fixed.
\item[$f_a$:] If $x > 0$, then add the string $(1,-1)$ to $(\nu,J)^{(a)}$.  Otherwise, let $\ell$ be the maximal length of all strings in $(\nu,J)^{(a)}$ which have label $x$.   Replace the string $(\ell, x)$ by the string $(\ell+1, x-1)$ and change all other labels so that all colabels remain fixed.
\end{itemize}
The remaining crystal structure on $\RC(\infty)$ is given by
\begin{subequations}\label{epphiwt}
\begin{align}
\label{Binf_ep} \varepsilon_a(\nu,J) &= \max\{ k \in \ZZ_{\ge0} \mid  e_a^k(\nu,J) \neq 0 \}, \\ 
\label{Binf_phi} \varphi_a(\nu,J) &= \varepsilon_a(\nu,J) + \langle h_a,\wt(\nu,J)\rangle, \\ 
\label{Binf_wt} \wt(\nu,J) &= -\sum_{(a,i)\in\HH} im_i^{(a)}\alpha_a = -\sum_{a\in I} |\nu^{(a)}|\alpha_a.
\end{align}
\end{subequations}
\end{dfn}

It is worth noting that, in this case, the definition of the vacancy numbers reduces to
\begin{equation}
p_i^{(a)}(\nu) = p_i^{(a)} = -\sum_{(b,j) \in \HH_0} A_{ab} \min(i, j) m_j^{(b)}.
\end{equation}

\begin{ex}
Rigged configurations will be shown as a sequence of partitions where the vacancy numbers will be written on the left and the corresponding rigging on the right. Let $\g$ be of type $A_5$ and $(\nu,J) = f_5 f_4 f_5 f_2 f_1 f_2 f_3(\nu_{\emptyset}, J_{\emptyset})$ be the rigged configuration
\begin{align*}
(\nu,J) &= 
\begin{tikzpicture}[scale=.35,baseline=-18]
 \rpp{1}{-1}{-1}
\begin{scope}[xshift=5cm]
 \rpp{2}{-1}{-2}
\end{scope}
\begin{scope}[xshift=12cm]
 \rpp{1}{1}{0}
\end{scope}
\begin{scope}[xshift=16cm]
 \rpp{1}{-1}{0}
\end{scope}
\begin{scope}[xshift=21cm]
 \rpp{2}{-1}{-3}
\end{scope}
\end{tikzpicture}
\intertext{Then $\wt(\nu,J) = -\alpha_1 - 2\alpha_2 - \alpha_3 - \alpha_4 - 2\alpha_5$,}
e_2(\nu,J) &= 
\begin{tikzpicture}[scale=.35,baseline=-18]
 \rpp{1}{-1}{-1}
\begin{scope}[xshift=5cm]
 \rpp{1}{0}{0}
\end{scope}
\begin{scope}[xshift=12cm]
 \rpp{1}{1}{0}
\end{scope}
\begin{scope}[xshift=16cm]
 \rpp{1}{-1}{0}
\end{scope}
\begin{scope}[xshift=21cm]
 \rpp{2}{-1}{-3}
\end{scope}
\end{tikzpicture}
\intertext{and}
f_2(\nu,J) &= 
\begin{tikzpicture}[scale=.35,baseline=-18]
 \rpp{1}{-1}{-1}
\begin{scope}[xshift=5cm]
 \rpp{3}{-2}{-4}
\end{scope}
\begin{scope}[xshift=12cm]
 \rpp{1}{1}{0}
\end{scope}
\begin{scope}[xshift=16cm]
 \rpp{1}{-1}{0}
\end{scope}
\begin{scope}[xshift=21cm]
 \rpp{2}{-1}{-3}
\end{scope}
\end{tikzpicture}
\end{align*}
\end{ex}

\begin{thm}[\cite{SS2015}]
\label{thm:RCinf_nsl}
The map defined by $(\nu_\emptyset,J_\emptyset) \mapsto u_\infty$, where $u_\infty$ is the highest weight element of $B(\infty)$, is a $U_q(\g)$-crystal isomorphism $\RC(\infty) \cong B(\infty)$.
\end{thm}

We can extend the crystal structure on rigged configurations to model $B(\lambda)$ as follows. Consider a multiplicity array $L$ such that $\lambda = \sum_{a \in I} L^{(a)} \Lambda_a$. We first we modify the definition of the weight to be $\wt'(\nu, J) = \wt(\nu, J) + \lambda$. Next, modify the crystal operators by saying $f_a(\nu, J) = 0$ if $\varphi_a(\nu, J) = 0$. Equivalently, we say $f_a(\nu, J) = 0$ if the result under $f_a$ above is not a valid rigged configuration. Let $\RC(\lambda)$ denote the closure of $(\nu_{\emptyset}, J_{\emptyset})$ under these modified crystal operators. This arises from the natural projection of $B(\infty) \longrightarrow B(\lambda)$.

\begin{thm}[\cite{SS15}]
Let $\g$ be of finite type. Then
$
\RC(\lambda) \iso B(\lambda).
$
\end{thm}

\begin{figure}[ht]
\input{./B2-21-RC.tex}
\caption{The crystal graph $\RC(\Lambda_1+2\Lambda_2)$ of type $B_2$, created using Sage.}
\end{figure}
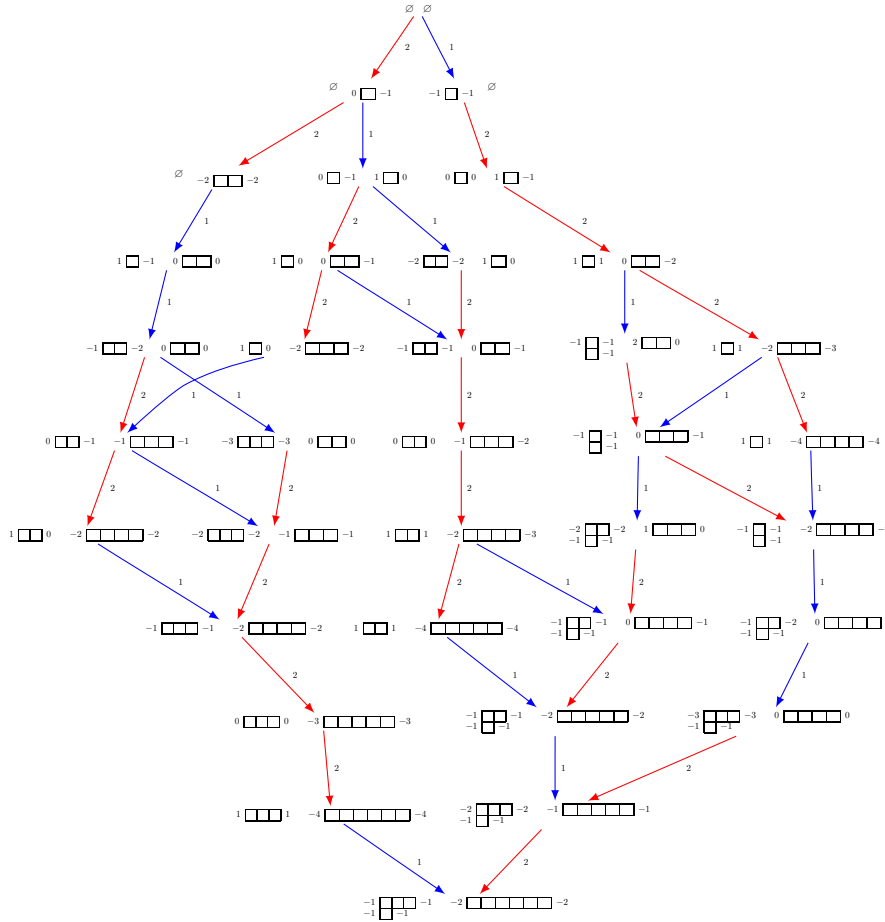

\subsection{Bijection with tableau model for highest weight crystals}

Kirillov-Reshetikhin (KR) crystals are crystals $B^{r,s}$ associated to certain finite-dimensional $U_q([\widetilde{\g}, \widetilde{\g}])$-modules $W^{r,s}$, where $r$ is a node in the Dynkin diagram and $s$ is a positive integer. 
As $U_q(\g)$-crystals, KR crystals have the direct sum decompositions
\[
B^{r,s} = B(s\Lambda_r) \oplus \bigoplus_{\lambda} B(\lambda),
\]
where the sum is over $\lambda \in P^+$ satisfying certain conditions~\cite{Chari01, FOS09, HKOTT02,HKOTY99, Hernandez06,Hernandez10, Nakajima03}. We note that we only work with the classical crystal structure of $B^{r,s}$, and as such, we simply consider $B^{r,s}$ to be a $U_q(\g)$-crystal.

In~\cite{OSS03}, a bijection $\Phi$ from classically highest weight elements in a tensor product of KR crystals of the form $(B^{1,1})^{\otimes N}$ for all non-exceptional affine types was described.  A similar bijection in type $E_6^{(1)}$ and $D_4^{(3)}$ was given in~\cite{OS12} and~\cite{Scrimshaw15}, respectively. For simplicity, if $B = \bigotimes_{i=1}^N B^{r_i,1}$ we write $\RC(B) = \RC(\lambda)$, where $\lambda = \sum_{a \in I} c_a \Lambda_a$ with $c_a$ equal to the number of factors $B^{a,1}$ occurring in $B$.

\begin{remark}
\label{remark:KR_crystal difference}
We define
\[
\widetilde{B}^{r,1} = \begin{cases}
B^{n,2} & \g = B_n \text{ and } r = n, \\
B^{n,1} \otimes B^{n+1,1} & \g = D_{n+1} \text{ and } r = n, \\
B^{r,1} & \text{otherwise.}
\end{cases}
\]
Because we will only use $\widetilde{B}^{r,s}$ for the remainder of this paper and to ease the burden of notation, we will simply write $B^{r,s}$. We also note that this allows us to not consider any special modifications to $\Phi$ as in~\cite{SS15}.
\end{remark}

The bijection $\Phi$ is given by applying the basic algorithm given in~\cite{OSS03}
\[
\delta \colon \RC(B^{1,1} \otimes B^*) \longrightarrow \RC(B^*)
\] 
as many times as necessary, where $B^* = \bigotimes_{i=1}^N B^{r_i,1}$. The algorithm $\delta$ is given by traversing the crystal graph $\TT(\Lambda_1)$ (of type $\g$) starting at $\boxed{1} \in \TT(\Lambda_1)$, where for each edge $a$ we remove a box from a singular string from $\nu^{(a)}$ of strictly longer length than the previously selected string after removal. For $\widetilde{\g}$ of type $D_{n+1}^{(1)}$, we choose the smaller singular string between $\nu^{(n)}$ and $\nu^{(n+1)}$ when we are at $\boxed{n-1} \in \TT(\Lambda_1)$.  If we cannot find a singular string or there are no outgoing arrows when we are at $\boxed{r} \in \TT(\Lambda_1)$, then we say $\delta$ returns $r$ and we make all changed strings singular.

We also have the following modification for $\widetilde{\g}$ of type $D_{n+1}^{(2)}$.  Let $\ell^{(a)}$ denote the original length of the selected strings in $\nu^{(a)}$ (for $a < r$, we have $\ell^{(a)} = 1$). We say a string $(i, x)$ is \emph{quasi-singular} if $x = p_i^{(n)} - 1$ and there does not exist a singular string of length $i$. For $\nu^{(n)}$, if the singular string has length $1$, we immediately return $\emptyset \in B^{1,1}$.  Otherwise we look for the smallest string of longer length than the previously which is either
\begin{itemize}
\item[(S)] singular,
\item[(Q)] quasi-singular.
\end{itemize}
If no such string exists, we return $n-1$ (as usual). If we are in case~(S), we remove 2 boxes from the singular string and proceed from $\boxed{\overline{n-1}} \in \TT(\Lambda_1)$.  If we are in case~(Q), we remove a box from the quasi-singular string and look for a larger singular string in $\nu^{(n)}$. If no such string exists, we return $0$. Otherwise we say we are in case~(Q,S) and remove a box from the found singular string. We then continue from $\boxed{\overline{n-1}} \in \TT(\Lambda_1)$. If we are at $\boxed{\bi} \in \TT(\Lambda_1)$ and the length of the previously selected string before removal equals $\ell^{(a)}$, we remove a second box from the string originally selected in $\nu^{(a)}$.

After all boxes are removed, we make all of the changed strings singular unless we are in case~(Q,S), in which case the (longer) selected singular string in $\nu^{(n)}$ is made quasi-singular.  

\begin{ex}
Consider $B = (B^{1,1})^{\otimes 5}$ of type $D_5^{(2)}$. Therefore applying $\delta$ each time, we have
\[
\begin{tikzpicture}[scale=.35]
\begin{scope}[yshift=1cm]
 \rpp{2,1,1}{0,0,0}{0,1,1}
\begin{scope}[xshift=5cm]
 \rpp{2,1}{0,0}{0,0}
\end{scope}
\begin{scope}[xshift=10cm]
 \rpp{2}{1}{1}
\end{scope}
\begin{scope}[xshift=15cm]
 \rpp{2}{0}{0}
\end{scope}
\end{scope} 
\draw[->] (8,-3cm) -- (8,-6cm) node[midway,right] {$\bon$};
\begin{scope}[yshift=-6cm]
 \rpp{1,1}{0,0}{1,1}
\begin{scope}[xshift=5cm]
 \rpp{1}{0}{0}
\end{scope}
\begin{scope}[xshift=10cm]
 \draw (0.5,-1.5) node {$\emptyset$};
\end{scope}
\begin{scope}[xshift=15cm]
 \draw (0.5,-1.5) node {$\emptyset$};
\end{scope}
\end{scope} 
\draw[->] (8,-9cm) -- (8,-12cm) node[midway,right] {$1$};
\begin{scope}[yshift=-12cm]
 \rpp{1,1}{0,0}{0,0}
\begin{scope}[xshift=5cm]
 \rpp{1}{0}{0}
\end{scope}
\begin{scope}[xshift=10cm]
 \draw (0.5,-1.5) node {$\emptyset$};
\end{scope}
\begin{scope}[xshift=15cm]
 \draw (0.5,-1.5) node {$\emptyset$};
\end{scope}
\end{scope} 
\draw[->] (8,-15cm) -- (8,-18cm) node[midway,right] {$3$};
\begin{scope}[yshift=-18cm]
 \rpp{1}{0}{0}
\begin{scope}[xshift=5cm]
 \draw (0.5,-1.5) node {$\emptyset$};
\end{scope}
\begin{scope}[xshift=10cm]
 \draw (0.5,-1.5) node {$\emptyset$};
\end{scope}
\begin{scope}[xshift=15cm]
 \draw (0.5,-1.5) node {$\emptyset$};
\end{scope}
\end{scope} 
\draw[->] (8,-21cm) -- (8,-24cm) node[midway,right]{$2$};
\begin{scope}[yshift=-24cm]
 \draw (0.5,-1.5) node {$\emptyset$};
\begin{scope}[xshift=5cm]
 \draw (0.5,-1.5) node {$\emptyset$};
\end{scope}
\begin{scope}[xshift=10cm]
 \draw (0.5,-1.5) node {$\emptyset$};
\end{scope}
\begin{scope}[xshift=15cm]
 \draw (0.5,-1.5) node {$\emptyset$};
\end{scope}
\end{scope} 
\draw[->] (8,-27cm) -- (8,-30cm) node[midway,right] {$1$};
\begin{scope}[yshift=-30cm]
 \draw (0.5,-1.5) node {$\emptyset$};
\begin{scope}[xshift=5cm]
 \draw (0.5,-1.5) node {$\emptyset$};
\end{scope}
\begin{scope}[xshift=10cm]
 \draw (0.5,-1.5) node {$\emptyset$};
\end{scope}
\begin{scope}[xshift=15cm]
 \draw (0.5,-1.5) node {$\emptyset$};
\end{scope}
\end{scope} 
\end{tikzpicture}
\]
(where the result from each application of $\delta$ is to the right of the arrow) and resulting in
\[
\young(\bon) \otimes \young(1) \otimes \young(3) \otimes \young(2) \otimes \young(1)\ .
\]
\end{ex}

We can now extend this bijection to $\bigotimes_{i=1}^N B^{r_i,1}$ as given for types $A_n^{(1)}$~\cite{KSS1999}, $D_{n+1}^{(1)}$~\cite{S05}, $A_{2n-1}^{(2)}$~\cite{SS15}, $D_{n+1}^{(2)}$~\cite{OSS03II}, and $D_4^{(3)}$ \cite{Scrimshaw15} by considering the map 
\[
\mathrm{lt} \colon \RC(B^{r,1} \otimes B^*) \longrightarrow \RC(B^{1,1} \otimes B^{r-1,1} \otimes B^*)
\] 
for $r \geq 2$ which adds a singular string of length 1 to all $\nu^{(a)}$ for $a < r$ and then applying $\delta$. We can combine these two steps $\delta' := \delta \circ \mathrm{lt}$ where we just begin $\delta$ starting at $\boxed{r} \in \TT(\Lambda_1)$ (i.e., the first box we try to remove is in $\nu^{(r)}$). Unless otherwise noted, we will be using $\delta'$ in place of $\delta$.

\begin{ex}
Consider $B^{1,1} \otimes B^{2,1} \otimes B^{1,1}$ of type $A_5^{(2)}$. Therefore applying $\delta$ each time, we have
\[
\begin{tikzpicture}[scale=.35]
 \rpp{1,1}{0,0}{0,0}
\begin{scope}[xshift=6cm]
 \rpp{1,1}{1,0}{1,1}
\end{scope}
\begin{scope}[xshift=12cm]
 \rpp{1}{0}{0}
\end{scope}
\draw[->] (6,-4cm) -- (6,-7cm) node[midway,right] {$\bth$};
\begin{scope}[yshift=-7cm]
 \rpp{1}{0,0}{0}
\begin{scope}[xshift=6cm]
 \rpp{1}{1,0}{1}
\end{scope}
\begin{scope}[xshift=12cm]
 \draw (0.5,-1.5) node {$\emptyset$};
\end{scope}
\end{scope} 
\draw[->] (6,-10cm) -- (6,-13cm) node[midway,right] {$3$};
\begin{scope}[yshift=-13cm]
 \rpp{1}{0}{0}
\begin{scope}[xshift=6cm]
 \rpp{1}{0}{0}
\end{scope}
\begin{scope}[xshift=12cm]
 \draw (0.5,-1.5) node {$\emptyset$};
\end{scope}
\end{scope} 
\draw[->] (6,-16cm) -- (6,-19cm) node[midway,right] {$2$};
\begin{scope}[yshift=-19cm]
 \rpp{1}{0}{0}
\begin{scope}[xshift=6cm]
 \draw (0.5,-1.5) node {$\emptyset$};
\end{scope}
\begin{scope}[xshift=12cm]
 \draw (0.5,-1.5) node {$\emptyset$};
\end{scope}
\end{scope} 
\draw[->] (6,-22cm) -- (6,-25cm) node[midway,right] {$1$};
\begin{scope}[yshift=-25cm]
 \draw (0.5,-1.5) node {$\emptyset$};
\begin{scope}[xshift=6cm]
 \draw (0.5,-1.5) node {$\emptyset$};
\end{scope}
\begin{scope}[xshift=12cm]
 \draw (0.5,-1.5) node {$\emptyset$};
\end{scope}
\end{scope} 
\end{tikzpicture}
\]
and resulting in
\[
\young(\bth) \otimes \young(3,2) \otimes \young(1)\ .
\]
\end{ex}

For $\delta^{-1}(b)$, in general we select the largest singular strings starting at $\boxed{b} \in \TT(\Lambda_1)$ and following arrows in reverse until we reach $\boxed{r} \in \TT(\Lambda_1)$.

\begin{remark}
We note that for $B(\Lambda_r) \subseteq B^{r,1}$, the classical tableaux representation is the same as the Kirillov--Reshetikhin tableaux representation of~\cite{S05,OSS13,SS15,Scrimshaw15}.
\end{remark}

We also have the following from~\cite[Prop.~6.4]{SS15} by using the results of~\cite{DS06,Sakamoto13,OSS03II,SS15}.

\begin{thm}
\label{thm:iso_columns}
Let $B = \bigotimes_{i=1}^N B^{r_i,1}$ be of type $\widetilde{\g}$. Then $\Phi$ is a $U_q(\g)$-crystal isomorphism, where $\widetilde{\g}$ and $\g$ are related via Table \ref{table:affine}.
\end{thm}

Given dominant integral weight $\lambda = \sum_{a\in I} c_a\Lambda_a \in \widehat{P}^+$, define
\[
B^{\otimes \lambda} = \bigotimes_{a \in I} (B^{a,1})^{\otimes c_a},
\]
where the factors are ordered to be weakly decreasing with respect to $a$.

\begin{ex}
If $\lambda = 2\Lambda_4 + 3\Lambda_1$ in type $A_{17}$, then $B^{\otimes \lambda} = (B^{4,1})^{\otimes 2} \otimes (B^{1,1})^{\otimes 3}$.
\end{ex}

Next we can restrict $\Phi$ to a (classical, or $U_q(\g)$-)crystal isomorphism between $\RC(\lambda)$ and $\TT(\lambda)$ as follows. We recall that there exists, by weight considerations, a unique copy
\[
B\left( \sum_{i=1}^N \mu_i \right) \subseteq \bigotimes_{i=1}^N B(\mu_i)
\]
generated by $u_{\mu_1} \otimes \cdots \otimes u_{\mu_N}$, where $u_{\mu_i} \in B(\mu_i)$ is the unique highest weight element. Then there exists a unique embedding
\begin{equation}
\label{eq:subcrystal}
\TT(\lambda) \subseteq \bigotimes_{a \in I} \TT(\Lambda_a)^{\otimes c_a} \subseteq B^{\otimes \lambda}.
\end{equation}
Because we have chosen the ordering of $B^{\otimes \lambda}$ to be in decreasing order, the natural embedding of $\TT(\lambda) \longhookarrow \TT(\Lambda_1)^{\otimes \absval{\lambda}}$ agrees with the natural classical embedding of $B^{\otimes \lambda} \longhookarrow \TT(\Lambda_1)^{\otimes \absval{\lambda}}$.
Moreover the highest weight element $T_{\lambda} \in \TT(\lambda)$ is given by (a tensor product of) basic columns. So in $\RC(B^{\otimes\lambda})$, considered as a $U_q(\g)$-crystal, the unique connected component generated by $(\nu_{\emptyset}, J_{\emptyset})$ is $\RC(\lambda)$. Therefore $\Phi(\nu_{\emptyset}, J_{\emptyset}) = T_{\lambda}$, and hence we have the following.

\begin{prop}
\label{prop:classical_bijection}
The crystal isomorphism $\Phi \colon \RC(B^{\otimes \lambda}) \longrightarrow B^{\otimes \lambda}$ restricts to a crystal isomorphism between $\RC(\lambda)$ and $\TT(\lambda)$.
\end{prop}

\section{The crystal isomorphism between $\RC(\infty)$ and $\TT(\infty)$}
\label{sec:isomorphism}

Let $\g$ be of type $A_n$, $B_n$, $C_n$, $D_{n+1}$, or $G_2$. By knowing to which $B(\lambda)$ a particular rigged configuration belongs (in fact, there are infinite such $\lambda$), we can extend the bijection between rigged configurations and tensor products of KR crystals to $B(\infty)$ by projecting down to $B(\lambda)$. We show that this implies the (induced) map given by lifting the isomorphism $\Phi \colon \RC(\lambda) \longrightarrow \TT(\lambda)$ from Proposition~\ref{prop:classical_bijection} is an isomorphism between $\RC(\infty)$ and $\TT(\infty)$ (and thus the unique isomorphism between $\RC(\infty)$ and $\TT(\infty)$ since the automorphism group of $B(\infty)$ is the trivial group). 

Consider
\[
\RC^V = \bigsqcup_{\lambda \in P^+} \RC(\lambda),
\]
and denote $(\nu, J, \lambda) \in \RC^V$ as the element $(\nu, J) \in \RC(\lambda)$. Define an equivalence relation on $\RC^V$ by asserting
\begin{equation}
\label{eq:equivalence_class}
(\nu, J, \lambda) \sim (\nu', J', \lambda') \iff \nu = \nu' \text{ and } J = J'.
\end{equation}
Note the vacancy numbers will vary over the equivalence class.  Equivalently, we have defined a subset $W_{(\nu,J)} \subseteq P^+$ such that $(\nu, J) \in \RC(\lambda)$ for all $\lambda \in W_{(\nu,J)}$. We show that every equivalence class of large tableaux embeds into an equivalence class of $\RC^V$ and that $\Phi$ induces a bijection from $\RC^V / \sim$ to $\TT(\infty) = \TT^L / \sim$. Subsequently, we show that this induced bijection is the desired crystal isomorphism.

For a sequence of partitions $\nu = (\nu^{(a)})_{a \in I}$, define $\lambda_{\nu} \in P^+$ by
\[
\lambda_{\nu} := \sum_{\substack{a \in I \\ a < n}} \big( \lvert \nu^{(a)} \rvert + 1 \big) \Lambda_a + \lambda_{\nu}^{(n)},
\]
where
\[
\lambda_{\nu}^{(n)} := \begin{cases}
2 \big( \lvert \nu^{(n)} \rvert + 1 \big) \Lambda_n & \g = B_n,
\\ (\max\bigl( \lvert \nu^{(n)} \rvert, \lvert \nu^{(n+1)} \rvert \bigr) + 1 \big) (\Lambda_n + \Lambda_{n+1}) & \g = D_{n+1},
\\ \big( \lvert \nu^{(n)} \rvert + 1 \big) \Lambda_n & \text{otherwise,}
\end{cases}
\]
and
\[
\RC^{EV} = \{ (\nu, J, \lambda) \in \RC^V \mid \lambda \geq \lambda_{\nu} \}.
\]
Here, $\lambda \leq \mu$ means $\lambda = (\lambda_i : i \in I)$ and $\mu = (\mu_i \mid i \in I)$ with $\lambda_i \le \mu_i$ for all $i \in I$.
We can restrict the equivalence relation given in equation \eqref{eq:equivalence_class} to $\RC^{EV}$ (so that classes in $\RC^{EV}$ are subclasses of those in $\RC^V$).  Call a rigged configuration \emph{extra valid} if it belongs to $\RC^{EV}$, and call a rigged configuration \emph{marginally extra valid} if $(\nu, J) \in \RC(\lambda_{\nu})$. We note that for each equivalence class in $\RC^{EV}$, there is a unique marginally extra valid rigged configuration because it has the smallest possible vacancy numbers.

\begin{lemma}
If $(\nu, J) \in \RC(\infty)$, then $(\nu, J) \in \RC(\lambda_{\nu})$.
\end{lemma}

\begin{proof}
This clearly holds for $(\nu_{\emptyset}, J_{\emptyset}) \in \RC(0)$. We will now proceed by induction by applying $f_a$ for some $a \in I$. Suppose $(\nu, J) \in \RC(\lambda_{\nu})$, we will show that $(\nu',J') = f_a(\nu, J)$ is in $\RC(\lambda_{\nu'})$. We note that the only possibile failure will occur if $x' > p_{i+1}^{(a)}(\nu')$ for the string $(i, x)$ acted on by $f_a$ since all other colabels remain fixed. We have
\[
p_{i+1}^{(a)}(\nu') - p_{i+1}^{(a)}(\nu) = -1
\]
since $\lambda_{\nu'} - \lambda_{\nu} = \Lambda_a$.  But because $x' - x = -1$, we must have $x' \leq p_{i+1}^{(a)}(\nu')$. Therefore $(\nu', J') \in \RC(\lambda_{\nu'})$.  Since there is some path to $(\nu_{\emptyset}, J_{\emptyset})$, the proof follows by induction.
\end{proof}

\begin{lemma}
\label{lemma:embed_MLT}
Let $T \in \TT^L$.  Then $\Phi^{-1}(T) \in \RC^V$.  Moreover, if $t \in [T]$, then $\Phi^{-1}(t) \in [\Phi^{-1}(T)]$.
\end{lemma}

\begin{proof}
Fix a large tableau $T$. By the definition of $\Phi$ and $\RC^V$, we have $\Phi^{-1}(T) \in \RC^V$. Next we note a column in $T$ of height $r$ has the form
\begin{equation}\label{column}
\begin{tikzpicture}[baseline]
\matrix [matrix of math nodes,column sep=-.4, row sep=-.5,text height=10,text width=23,align=center,inner sep=3] 
 {
 	\node[draw,fill=gray!30]{1}; \\
  	\node[draw,fill=gray!30]{2}; \\
  	\node[draw,fill=gray!30]{\vdots}; \\ 
  	\node[draw,fill=gray!30]{r-1}; \\
	\node[draw]{x}; \\
 };
\end{tikzpicture}
\end{equation}
where $1 \le r \le n$. We are going to add a column of the form above in $B^{r,1}$. Suppose $B = \bigotimes_{i=1}^N B^{r_i, s_i}$ and $(\nu,J) \in \RC(B^{a-1,1} \otimes B)$ for $a < r$.  Applying $\delta^{-1}$ to the column in \eqref{column} will change $(\nu,J)$, and the order by which $(\nu,J)$ is affected is determined by reading the column from top to bottom. Indeed, applying $\delta^{-1}$ corresponding to the $a$-box of the column will add $1$ to the vacancy numbers of $\nu^{(a)}$ and subtract $1$ from the vacancy numbers of $\nu^{(a-1)}$ if $a > 1$.  Now suppose we are performing $\delta^{-1}$ corresponding to the $x$-box at the bottom of the column.  Then this application of $\delta^{-1}$ can only add boxes to $\nu^{(a)}$ for $a \geq r$, and can at most decrease the vacancy numbers in $\nu^{(r-1)}$ by $1$. In particular, if $x = r$, then the net result of adding this column is that the vacancy numbers of $\nu^{(r)}$ increased by $1$ and the vacancy numbers of $\nu^{(a)}$ for $a<r$ are left unchanged.

In applying $\Phi^{-1}$ to $T$, we are moving from right to left in $T$, so we are applying $\delta^{-1}$ to columns weakly increasing in height. Moreover, $t$ differs from $T$ by the, without loss of generality, addition of columns with $x = r$. Therefore from the above, we must have $\Phi^{-1}(t) \sim \Phi^{-1}(T)$ for any $t \in [T]$.
\end{proof}

\begin{lemma}
\label{lemma:embed_RC}
Let $(\nu, J) \in \RC^{EV}$. 
Then $\Phi(\nu,J) \in \TT^L$. Moreover if $(\nu',J') \in [(\nu,J)]$, then $\Phi(\nu',J') \in [\Phi(\nu,J)]$.
\end{lemma}

\begin{proof}
Fix some extra valid rigged configuration $(\nu, J)$. Suppose $\Phi(\nu, J) \notin \TT^L$. Therefore during the procedure of $\Phi$ in a column of height $r$ at height $a < r$, we return $x \geq a$, so we remove at least one box from $\nu^{(a)}$. Therefore we must remove at least
\[
1 + \inner{\alpha_a}{\lambda_{\nu}} = 1 + 1 + |\nu^{(a)}|
\]
boxes from $\nu^{(a)}$ since we must return at least $x$ by the semistandard condition. This is a contradiction, and so we must return $a$. Similarly for the left-most column of height $r$, we must return $r$. Hence $\Phi(\nu, J) \in \TT^L$.

Next we show $\Phi(\nu', J') \in [\Phi(\nu, J)]$. Consider $(\nu', J') \in \RC(\lambda_{\nu'})$ such that $\lambda_{\nu'} - \lambda_{\nu} = \Lambda_r$ for some $1 \leq r \leq n$. We will show that $\Phi(\nu', J')$ differs from $\Phi(\nu, J)$ by a basic column of height $r$. We note that
\begin{equation}
\label{eq:diff_vac_nums}
p_i^{(a)}(\nu') - p_i^{(a)}(\nu) = \delta_{ar},
\end{equation}
and therefore all columns of height at least $r+1$ are equal under $\Phi$. That is $\delta$ returns the same elements on both $(\nu', J')$ and $(\nu, J)$. Furthermore, once we've removed all such columns (there are the same number of columns in each), the results are equivalent such that the difference of the weights is still $\Lambda_r$. Hence Equation~\eqref{eq:diff_vac_nums} still holds.

Now we have one additional column of height $r$ in $\lambda_{\nu'}$. From Equation~\eqref{eq:diff_vac_nums}, we must have all strings of $\nu'^{(r)}$ being non-singular. Therefore $\delta$ returns $r$, and we increase all vacancy numbers of $\nu'^{(r-1)}$. Thus all strings of $\nu'^{(r-1)}$ are non-singular and iterating this, we see that we return a basic column of height $r$.

At this point, the resulting rigged configurations are equal (not just equivalent as they have the same weight), and hence the remaining result from $\Phi$ are equal. Since there exists a unique element of minimal weight in $[(\nu, J)]$, the claim follows by induction.
\end{proof}

The following lemma is analogous to~\cite[Lemma~3.2]{HL08}, which shows that the crystal operators are well-defined on equivalence classes.

\begin{lemma}
\label{lemma:preserve_Kashiwara}
Fix $a\in I$.
\begin{enumerate}
\item If $(\nu,J) \in \RC^{EV}$, then $f_a(\nu,J) \neq 0$.
\item Given any element of $\RC(\infty)$, we can always find a representative $(\nu,J) \in \RC^{EV}$ such that $f_a(\nu,J)$ is a valid rigged configuration.
\item If $(\nu,J)$ and $(\nu',J')$ are in the same equivalence class in $\RC^V/\sim$, then $[f_a(\nu,J)] = [f_a(\nu',J')]$.
\item If $(\nu,J)$ is valid, then $e_a(\nu,J)$ is either valid or zero.
\item If $(\nu,J)$ and $(\nu',J')$ are in the same equivalence class in $\RC^V/\sim$, then either $[e_a(\nu,J)] = [e_a(\nu',J')]$ or both $e_a(\nu,J) = e_a(\nu',J') = 0$.
\end{enumerate}
\end{lemma}

\begin{proof}
These statements can be seen directly from the definitions.
\end{proof}

Thus we can define
\begin{subequations}
\label{eq:crystal_structure}
\begin{align}
e_a [(\nu, J)] & = [e_a(\nu, J)]
\\ f_a [(\nu, J)] & = [f_a(\nu, J)]
\\ \wt [(\nu, J)] & = \sum_{a \in I} \lvert \nu^{(a)} \rvert \Lambda_a,
\\ \varepsilon_a [(\nu, J)] & = \max\{ e_a^k(\nu, J) \neq 0 \mid k \in \ZZ_{>0} \},
\\ \varphi_a [(\nu ,J)] & = \varepsilon_a [(\nu, J)] + \inner{\wt [(\nu, J)]}{h_a},
\end{align}
\end{subequations}
for any $[(\nu, J)] \in \RC^V / \sim$ with appropriate representative $(\nu, J)$. Therefore a straightforward check shows the following.

\begin{prop}
Equation~\eqref{eq:crystal_structure} defines an abstract $U_q(\g)$-crystal structure on $\RC^V / \sim$.
\end{prop}

Define a map $\Psi\colon \RC(\infty) \longrightarrow \TT(\infty)$ using the sequence of maps
\[
\begin{array}{ccccccccc}
\RC(\infty) &\longdoublearrow& \RC(\lambda_\nu) &\longhookarrow& B^{\otimes \lambda_\nu} &\longdoublearrow& \TT(\lambda_\nu) &\longhookarrow& \TT(\infty), \\
(\nu,J) &\mapsto& (\nu,J) &\mapsto& (\nu,J) &\mapsto& \Phi(\nu,J) &\mapsto& \Phi(\nu,J).
\end{array}
\]
Conversely, for $T\in \TT(\infty)$, let $\lambda_T \in P^+$ partition shape of $T$.  There is a natural crystal embedding of a tableau $T$ into a tensor product of its columns in $B^{\otimes\lambda_T}$.  Denote the image of $T$ in $B^{\otimes\lambda_T}$ by $T^{\otimes\lambda_T
}$.  Now define a map $\Xi\colon \TT(\infty) \longrightarrow \RC(\infty)$ by the sequence of maps
\[
\begin{array}{ccccccccc}
\TT(\infty) &\longdoublearrow& \TT(\lambda_T) &\longhookarrow& B^{\otimes \lambda_T} &\longdoublearrow& \RC(\lambda_T)& \longhookarrow &\RC(\infty),\\
T &\mapsto& T &\mapsto& T^{\otimes\lambda_T} &\mapsto& \Phi^{-1}(T^{\otimes\lambda_T}) &\mapsto& \Phi^{-1}(T^{\otimes\lambda_T}).
\end{array}
\]

\begin{thm}
We have $\Xi\circ\Psi = \id_{\RC(\infty)}$ and $\Psi\circ\Xi = \id_{\TT(\infty)}$.
\end{thm}

\begin{proof}
Given a marginally large tableaux $T$ of shape $\lambda_T$, begin by projecting down to $B(\lambda_T)$. This preserves the tableaux $T$ and consider the natural embedding $T^{\prime}$ in $\bigotimes_i B^{r_i,1}$ given by Equation~\eqref{eq:subcrystal}. Next take $\Phi(T^{\prime})$, recall from Theorem~\ref{thm:iso_columns} that $\Phi$ is a bijection, and lift the resulting rigged configuration to $\RC(\infty)$. Last, we note that this procedure is well-defined over the equivalence class of large tableaux by Lemma~\ref{lemma:embed_MLT} and Lemma~\ref{lemma:embed_RC}, so this gives us the desired bijection.
\end{proof}

\begin{cor}
\label{cor:bijection}
The bijection $\Psi$ is a crystal isomorphism.
\end{cor}

\begin{proof}
This follows from the fact that $\Phi$ is a (classical) crystal isomorphism. Indeed, the map $\Psi$ is well-defined as a crystal morphism by Lemma~\ref{lemma:preserve_Kashiwara}. Let $(\nu,J) \in \RC(\infty)$ and $a\in I$. Next, denote the $e_a$ operator in the crystal $X$ by $e_a^X$. Then
\[
e_a^{\TT(\infty)}\Psi(\nu,J) = e_a^{\TT(\infty)}\Phi(\nu,J) = e_a^{\TT(\lambda_\nu)}\Phi(\nu,J) = e_a^{B^{\otimes\lambda_\nu}}(\nu,J) = e_a^{\RC(\lambda_\nu)}(\nu,J).
\]
Since $(\nu,J)$ is nonzero in $\RC(\lambda_\nu)$ by definition of $\lambda_\nu$, we have $e_a^{\RC(\lambda_\nu)}(\nu,J) = e_a^{\RC(\infty)}(\nu,J)$.  Thus $\Psi$ commutes with $e_a$.  Showing that $\Psi$ commutes with $f_a$ is similar.
\end{proof}

\begin{remark}
Consider some $T \in B(\infty)$, we can explicitly describe the image of $T$ when projecting under $p_{\lambda}$ to $B(\lambda)$ for all $\lambda \in P^+$. From this, we define an equivalence class $[T] = \{ p_{\lambda}(T) \mid p_{\lambda} \neq 0,\ \lambda \in P^+ \}$. We can check that $[T]$ corresponds to the class of all valid rigged configurations under $\Phi$.
\end{remark}

\begin{ex}
\label{ex:full_bijection}
Let $\g$ be of type $A_4$ and
\[
T = \begin{tikzpicture}[baseline]
\matrix [tab] 
 {
 	\node[draw,fill=gray!30]{1}; &  
	\node[draw,fill=gray!30]{1}; &
	\node[draw,fill=gray!30]{1}; & 
	\node[draw,fill=gray!30]{1}; & 
	\node[draw,fill=gray!30]{1}; & 
	\node[draw,fill=gray!30]{1}; & 
	\node[draw,fill=gray!30]{1}; & 
	\node[draw,fill=gray!30]{1}; & 
	\node[draw]{3}; & 
	\node[draw]{3}; \\
  	\node[draw,fill=gray!30]{2}; &  
	\node[draw,fill=gray!30]{2}; &  
	\node[draw,fill=gray!30]{2}; & 
	\node[draw,fill=gray!30]{2}; &
	\node[draw,fill=gray!30]{2}; & 
	\node[draw,fill=gray!30]{2}; & 
	\node[draw]{5}; \\
  	\node[draw,fill=gray!30]{3}; & 
  	\node[draw,fill=gray!30]{3}; &
	\node[draw]{4}; &
	\node[draw]{5}; &
	\node[draw]{5}; \\
  	\node[draw,fill=gray!30]{4}; \\
 };
\end{tikzpicture}.
\]
We first project onto $B(\lambda) \subseteq B^{\otimes \lambda}$ with $\lambda = \Lambda_4 + 4 \Lambda_3 + 2 \Lambda_2 + 3 \Lambda_1$ (which results in the same tableaux). Next we apply $\Phi \colon B^{\otimes \lambda} \longrightarrow \RC(B^{\otimes \lambda})$, and have the following:
\[
\begin{tikzpicture}[scale=.35]
 \rpp{1}{0}{0}
\begin{scope}[xshift=6cm]
 \rpp{1}{-1}{-1}
\end{scope}
\begin{scope}[xshift=12cm]
 \draw (0,-1.5) node {$\emptyset$};
\end{scope}
\begin{scope}[xshift=19cm]
 \draw (0,-1.5) node {$\emptyset$};
\end{scope}
\begin{scope}[xshift=25.5cm,yshift=-1.3cm]
\matrix [tab] 
 {
	\node[draw]{3}; \\
 };
\end{scope}
\begin{scope}[yshift=-3cm]
 \rpp{2}{0}{0}
\begin{scope}[xshift=6cm]
 \rpp{2}{-2}{-2}
\end{scope}
\begin{scope}[xshift=12cm]
 \draw (0,-1.5) node {$\emptyset$};
\end{scope}
\begin{scope}[xshift=19cm]
 \draw (0,-1.5) node {$\emptyset$};
\end{scope}
\begin{scope}[xshift=26cm,yshift=-1.3cm]
\matrix [tab] 
 {
	\node[draw]{3}; & 
	\node[draw]{3}; \\
 };
\end{scope}
\end{scope} 
\begin{scope}[yshift=-6cm]
 \rpp{2}{0}{1}
\begin{scope}[xshift=6cm]
 \rpp{2}{-2}{-2}
\end{scope}
\begin{scope}[xshift=12cm]
 \draw (0,-1.5) node {$\emptyset$};
\end{scope}
\begin{scope}[xshift=19cm]
 \draw (0,-1.5) node {$\emptyset$};
\end{scope}
\begin{scope}[xshift=26.5cm,yshift=-1.3cm]
\matrix [tab] 
 {
	\node[draw,fill=gray!30]{1}; & 
	\node[draw]{3}; & 
	\node[draw]{3}; \\
 };
\end{scope}
\end{scope} 
\begin{scope}[yshift=-9cm]
 \rpp{2}{0}{2}
\begin{scope}[xshift=6cm]
 \rpp{2,1}{-2,-1}{-2,-1}
\end{scope}
\begin{scope}[xshift=12cm]
 \rpp{1}{1}{1}
\end{scope}
\begin{scope}[xshift=19cm]
 \rpp{1}{-1}{-1}
\end{scope}
\begin{scope}[xshift=27cm,yshift=-1.5cm]
\matrix [tab] 
 {
 	\node[draw,fill=gray!30]{1}; & 
	\node[draw,fill=gray!30]{1}; & 
	\node[draw]{3}; & 
	\node[draw]{3}; \\
	\node[draw]{5}; \\
 };
\end{scope}
\end{scope} 
\begin{scope}[yshift=-12cm]
 \rpp{2}{0}{2}
\begin{scope}[xshift=6cm]
 \rpp{2,1}{-2,-1}{-1,0}
\end{scope}
\begin{scope}[xshift=12cm]
 \rpp{1}{1}{1}
\end{scope}
\begin{scope}[xshift=19cm]
 \rpp{1}{-1}{-1}
\end{scope}
\begin{scope}[xshift=27.5cm,yshift=-1.5cm]
\matrix [tab] 
 {
 	\node[draw,fill=gray!30]{1}; & 
 	\node[draw,fill=gray!30]{1}; & 
	\node[draw,fill=gray!30]{1}; & 
	\node[draw]{3}; & 
	\node[draw]{3}; \\
	\node[draw,fill=gray!30]{2}; & 
	\node[draw]{5}; \\
 };
\end{scope}
\end{scope} 
\begin{scope}[yshift=-16cm]
 \rpp{2}{0}{2}
\begin{scope}[xshift=6cm]
 \rpp{2,1}{-2,-1}{0,0}
\end{scope}
\begin{scope}[xshift=12cm]
 \rpp{2}{2}{2}
\end{scope}
\begin{scope}[xshift=19cm]
 \rpp{2}{-2}{-2}
\end{scope}
\begin{scope}[xshift=28cm,yshift=-1.8cm]
\matrix [tab] 
 {
	\node[draw,fill=gray!30]{1}; & 
 	\node[draw,fill=gray!30]{1}; & 
 	\node[draw,fill=gray!30]{1}; & 
	\node[draw,fill=gray!30]{1}; & 
	\node[draw]{3}; & 
	\node[draw]{3}; \\
	\node[draw,fill=gray!30]{2}; & 
	\node[draw,fill=gray!30]{2}; & 
	\node[draw]{5}; \\
	\node[draw]{5}; \\
 };
\end{scope}
\end{scope} 
\begin{scope}[yshift=-20cm]
 \rpp{2}{0}{2}
\begin{scope}[xshift=6cm]
 \rpp{2,1}{-2,-1}{0,0}
\end{scope}
\begin{scope}[xshift=12cm]
 \rpp{3}{2}{2}
\end{scope}
\begin{scope}[xshift=19cm]
 \rpp{3}{-3}{-3}
\end{scope}
\begin{scope}[xshift=28.5cm,yshift=-1.8cm]
\matrix [tab] 
 {
	\node[draw,fill=gray!30]{1}; & 
	\node[draw,fill=gray!30]{1}; & 
 	\node[draw,fill=gray!30]{1}; & 
 	\node[draw,fill=gray!30]{1}; & 
	\node[draw,fill=gray!30]{1}; & 
	\node[draw]{3}; & 
	\node[draw]{3}; \\
	\node[draw,fill=gray!30]{2}; & 
	\node[draw,fill=gray!30]{2}; & 
	\node[draw,fill=gray!30]{2}; & 
	\node[draw]{5}; \\
	\node[draw]{5}; & 
	\node[draw]{5}; \\
 };
\end{scope}
\end{scope} 
\begin{scope}[yshift=-24cm]
 \rpp{2}{0}{2}
\begin{scope}[xshift=6cm]
 \rpp{2,1}{-2,-1}{0,0}
\end{scope}
\begin{scope}[xshift=12cm]
 \rpp{4}{1}{1}
\end{scope}
\begin{scope}[xshift=19cm]
 \rpp{3}{-3}{-3}
\end{scope}
\begin{scope}[xshift=29cm,yshift=-1.8cm]
\matrix [tab] 
 {
	\node[draw,fill=gray!30]{1}; & 
	\node[draw,fill=gray!30]{1}; & 
 	\node[draw,fill=gray!30]{1}; & 
 	\node[draw,fill=gray!30]{1}; & 
	\node[draw,fill=gray!30]{1}; & 
	\node[draw,fill=gray!30]{1}; & 
	\node[draw]{3}; & 
	\node[draw]{3}; \\
	\node[draw,fill=gray!30]{2}; & 
	\node[draw,fill=gray!30]{2}; & 
	\node[draw,fill=gray!30]{2}; & 
	\node[draw,fill=gray!30]{2}; & 
	\node[draw]{5}; \\
	\node[draw]{4}; & 
	\node[draw]{5}; & 
	\node[draw]{5}; \\
 };
\end{scope}
\end{scope} 
\begin{scope}[yshift=-29cm]
 \rpp{2}{0}{2}
\begin{scope}[xshift=6cm]
 \rpp{2,1}{-2,-1}{0,0}
\end{scope}
\begin{scope}[xshift=12cm]
 \rpp{4}{1}{2}
\end{scope}
\begin{scope}[xshift=19cm]
 \rpp{3}{-3}{-2}
\end{scope}
\begin{scope}[xshift=30cm,yshift=-1.8cm]
\matrix [tab] 
 {
	\node[draw,fill=gray!30]{1}; & 
	\node[draw,fill=gray!30]{1}; & 
 	\node[draw,fill=gray!30]{1}; & 
 	\node[draw,fill=gray!30]{1}; & 
	\node[draw,fill=gray!30]{1}; & 
	\node[draw,fill=gray!30]{1}; & 
	\node[draw,fill=gray!30]{1}; & 
	\node[draw,fill=gray!30]{1}; & 
	\node[draw]{3}; & 
	\node[draw]{3}; \\
	\node[draw,fill=gray!30]{2}; & 
	\node[draw,fill=gray!30]{2}; & 
	\node[draw,fill=gray!30]{2}; & 
	\node[draw,fill=gray!30]{2}; & 
	\node[draw,fill=gray!30]{2}; & 
	\node[draw,fill=gray!30]{2}; & 
	\node[draw]{5}; \\
	\node[draw,fill=gray!30]{3}; & 
	\node[draw,fill=gray!30]{3}; & 
	\node[draw]{4}; & 
	\node[draw]{5}; & 
	\node[draw]{5}; \\
	\node[draw,fill=gray!30]{4}; \\
 };
\end{scope}
\end{scope} 
\end{tikzpicture}
\]
By mapping back the rigged configuration into $\RC(\infty)$, we obtain
\[
\Xi(T) = 
\begin{tikzpicture}[scale=.35,baseline=-18]
 \rpp{2}{0}{-1}
\begin{scope}[xshift=6cm]
 \rpp{2,1}{-2,-1}{-2,-2}
\end{scope}
\begin{scope}[xshift=12cm]
 \rpp{4}{1}{-2}
\end{scope}
\begin{scope}[xshift=19cm]
 \rpp{3}{-3}{-3}
\end{scope}
\end{tikzpicture}.
\]
In Sage, we can reproduce the example using
\begin{lstlisting}
sage: B = crystals.infinity.Tableaux("A4")
sage: t = B(rows=[[1,1,1,1,1,1,1,1,3,3],\
....:             [2,2,2,2,2,2,5],[3,3,4,5,5],[4]])
sage: RC = crystals.infinity.RiggedConfigurations("A4")
sage: defn = {B.highest_weight_vector():RC.highest_weight_vector()};
sage: Xi = B.crystal_morphism(defn)
sage: t.pp()
  1  1  1  1  1  1  1  1  3  3
  2  2  2  2  2  2  5
  3  3  4  5  5
  4
sage: ascii_art(Xi(t))
-1[ ][ ]0  -2[ ][ ]-2  -2[ ][ ][ ][ ]1  -3[ ][ ][ ]-3
           -2[ ]-1 
\end{lstlisting}
\end{ex}

\begin{ex}
Let $\g$ be of type $A_3$ and
\[
(\nu, J) = 
\begin{tikzpicture}[scale=.35,baseline=-18]
 \rpp{2}{-1}{2}
\begin{scope}[xshift=6cm]
 \rpp{3,1}{-1,-1}{1,2}
\end{scope}
\begin{scope}[xshift=13cm]
 \rpp{3}{-1}{0}
\end{scope}
\end{tikzpicture}.
\]
Then 
\[
\Psi(\nu, J) = \begin{tikzpicture}[baseline]
\matrix [tab] 
 {
 	\node[draw,fill=gray!30]{1}; & 
	\node[draw,fill=gray!30]{1}; & 
	\node[draw,fill=gray!30]{1}; & 
	\node[draw,fill=gray!30]{1}; & 
	\node[draw,fill=gray!30]{1}; & 
	\node[draw,fill=gray!30]{1}; & 
	\node[draw,fill=gray!30]{1}; & 
	\node[draw]{2}; & 
	\node[draw]{3}; \\
  	\node[draw,fill=gray!30]{2}; &  
	\node[draw,fill=gray!30]{2}; & 
	\node[draw,fill=gray!30]{2}; &
	\node[draw]{3}; &
	\node[draw]{4}; &
	\node[draw]{4}; \\
  	\node[draw,fill=gray!30]{3}; & 
	\node[draw]{4}; \\
 };
\end{tikzpicture}.
\]
\end{ex}

\section{Statistics}
\label{sec:statistics}

Consider a marginally large tableau $T \in \TT(\infty)$. A \emph{$k$-segment} of $T$ is a maximal sequence of $k$-boxes in $T$. Let $\seg'(T)$ be the total number of segments of $T$ that are not $i$-sequences in the $i$-th row. In other words, this is the total number of segments of $T$ minus the hieght of $T$.

\begin{dfn}[\cite{LS12,LS14}]
The \emph{segment statistic $\seg$} on marginally large tableaux is defined type-by-type as follows.
\begin{itemize}
\item[$A_n$:] Define $\seg(T) := \seg'(T)$.
\item[$B_n$:] Let $e_B(T)$ be the number of rows $i$ the contain both a $0$-segment and $\bi$ segment. Define $\seg(T) := \seg'(T) - e_B(T)$.
\item[$C_n$:] Define $\seg(T) := \seg'(T)$.
\item[$D_{n+1}$:] Let $e_D(T)$ be the number of rows $i$ that contain an $\bi$-segment but neither a $(n+1)$-segment nor $\overline{n+1}$-segment. Define $\seg(T) = \seg'(T) + e_D(T)$.
\item[$G_2$:] Let $e_G(T)$ be 1 if $T$ contains a $0$-segment and $\bon$-segment in the first row and 0 otherwise. Define $\seg(T) = \seg'(T) - e_G(T)$.
\end{itemize}
\end{dfn}

Define $\delta_{(r)}$ by $\delta$ on $\RC(\lambda)$ and then embedding into $\RC(\lambda - \Lambda_r)$ where $r = \max \{a \in I \mid \inner{h_a}{\lambda} \neq 0 \}$. Equivalently $\delta_{(r)} = \delta^r$ but returning only the $b$ from the first application of $\delta$.

\begin{dfn}
The \emph{repeat statistic $\rpt$} on rigged configurations is given recursively as follows. Consider $(\nu, J) \in \RC(\lambda_{\nu})$. Start with $r = n$ and $s = 0$. Let $x_r^c$ denote the smallest colabel of $\nu^{(r)}$ and project $(\nu, J)$ into $\RC(\lambda_{(r)})$, where $\lambda_{(r)} = \lambda - x_r^c \Lambda_r$. Let $c_r = \inner{h_r}{\lambda_{(r)}}$, and let $b^{(r)} = (b_1, \dotsc, b_{c_r})$ be the values returned by $\delta_{(r)}^{c_r}$ (in $\RC(\lambda')$). Increase $s$ by the number of distinct elements ocurring in $b$ (note any particular value in $b$ occurs sequentially). We also do the following modifications depending on the type:
\begin{itemize}
\item[$B_n$:] If $0, \overline{r} \in b$, subtract 1 from $s$.
\item[$D_{n+1}$:] If $\overline{r} \in b$ and $n+1, \overline{n+1} \notin b$, then add 1 to $s$.
\item[$G_2$:] If $0, \bon \in b$, subtract 1 from $s$
\end{itemize}
Now recurse with $r-1$ unless $r = 1$, and $\rpt$ is the final value of $s$.
\end{dfn}

From our definition of $\rpt$, for a fixed $r$ we have $\inner{h_a}{\lambda} = 0$ for all $a > r$. So the values $p_i^{(a)}$ are equal on $\RC(\infty)$ and $\RC(\lambda)$ for all $a > r$ and $i \in \ZZ$. So the map $\delta_{(r)}$ just starts each time at $\nu^{(r)}$ and is well defined since all strings in $\nu^{(a)}$ for $a < r$ are non-singular after applying $\delta$. Furthermore by Lemma~\ref{lemma:embed_RC}, the statistic $\rpt$ is well defined. Thus we have the following.

\begin{prop}
Let $(\nu, J) \in \RC(\infty)$. Then
\[
\rpt(\nu, J) = \seg(\Psi(\nu, J)).
\]
\end{prop}



From the definition of $\rpt$, it is clear that it is a direct translation of $\seg$ through the bijection $\Psi$. It would be good if there was a non-recursive translation of $\seg$ on rigged configurations.

\begin{ex}
Consider the rigged configuration $(\nu, J)$ obtained from Example~\ref{ex:full_bijection}:
\[
(\nu, J) =
\begin{tikzpicture}[scale=.35,baseline=-18]
 \rpp{2}{0}{2}
\begin{scope}[xshift=6cm]
 \rpp{2,1}{-2,-1}{0,0}
\end{scope}
\begin{scope}[xshift=12cm]
 \rpp{4}{1}{2}
\end{scope}
\begin{scope}[xshift=19cm]
 \rpp{3}{-3}{-2}
\end{scope}
\end{tikzpicture}.
\]
Thus we begin with $(\nu, J) \in \RC(\lambda_{\nu})$ with $\lambda_{\nu} = 4\Lambda_4 + 5\Lambda_3 + 4\Lambda_2 + 3\Lambda_1$, so we have
\[
\begin{tikzpicture}[scale=.35,baseline=-18]
 \rpp{2}{0}{2}
\begin{scope}[xshift=6cm]
 \rpp{2,1}{-2,-1}{1,1}
\end{scope}
\begin{scope}[xshift=12cm]
 \rpp{4}{1}{2}
\end{scope}
\begin{scope}[xshift=19cm]
 \rpp{3}{-3}{0}
\end{scope}
\end{tikzpicture}.
\]
We thus project onto $\RC(4\Lambda_3 + 3\Lambda_2 + 2\Lambda_1)$, and since $\inner{h_4}{\lambda_{(4)}} = 0$, there is nothing more to do. Next, after projecting onto $\RC(3\Lambda_3 + 3\Lambda_2 + 2\Lambda_1)$ and then applying $\delta_{(3)}^3$, we obtain
\[
\begin{tikzpicture}[scale=.35,baseline=-18]
 \rpp{2}{0}{2}
\begin{scope}[xshift=6cm]
 \rpp{2,1}{-2,-1}{0,1}
\end{scope}
\begin{scope}[xshift=13cm]
 \rpp{1}{1}{1}
\end{scope}
\begin{scope}[xshift=19cm]
 \rpp{1}{-1}{-1}
\end{scope}
\end{tikzpicture}
\]
with $b_{(3)} = (4,5,5)$. Projecting onto $\RC(\lambda_2 + 2\Lambda_1)$, we then obtain after applying $\delta_{(2)}$, we obtain
\[
\begin{tikzpicture}[scale=.35,baseline=-18]
 \rpp{2}{0}{0}
\begin{scope}[xshift=6cm]
 \rpp{2}{-2}{-2}
\end{scope}
\begin{scope}[xshift=13cm]
 \draw (0,-1.5) node {$\emptyset$};
\end{scope}
\begin{scope}[xshift=19cm]
 \draw (0,-1.5) node {$\emptyset$};
\end{scope}
\end{tikzpicture}
\]
with $b_{(2)} = (5)$. Finally, we apply $\delta_{(1)}^2$ and obtain $b_{(1)} = (3, 3)$. Therefore, we have $\rpt(\nu, J) = 2 + 1 + 1 = 4$. It is also easy to check that $\seg(T) = 4$.
\end{ex}

There is another statistic on rigged configurations which has a natural crystal interpretation. The \emph{difference statistic $\diff_a$} is defined by
\[
\diff_a(\nu, J) := \min_i \{ p_i^{(a)} - \max J_i^{(a)} \}
\]
for some $(\nu, J) \in \RC(\lambda)$. The difference statistic is measuring how far $(\nu, J)$ from being marginally valid in $\nu^{(a)}$, and so it can be interpreted as the largest $c_a$ such that $(\nu, J) \in \RC(\lambda - c_a \Lambda_a)$ (under the natural projection). We can also combine this into a single statistic $\diff(\nu, J) = \sum_{a \in I} \diff_a(\nu, J)$.

On highest weight crystals, the \emph{remove statistic $\rem_a$} is the largest $c_a$ such that $b \in B(\lambda)$ is non-zero under the natural projection to $B(\lambda - c_a \Lambda_a)$. Moreover, we can also combine this into a single statistic $\rem(\nu, J) = \sum_{a \in I} \rem_a(\nu, J)$. Given these interpretations, the following is an immediate consequence.

\begin{prop}
\label{prop:diff_rem_lambda}
Let $\Psi \colon \RC(\lambda) \longrightarrow B(\lambda)$ be an isomorphism. Then
\[
\diff_a(\nu, J) = \rem_a(\Psi(\nu, J)).
\]
\end{prop}

Furthermore, the difference statistic can be extended to $\RC(\infty)$, where it is measuring how far the rigged configuration is from being valid in $\nu^{(a)}$. Thus we see that the weight $\sum_{a \in I} \diff_a(\nu, J) \Lambda_a$ denotes the (unique) minimal weight that we need to project $(\nu, J) \in \RC(\infty)$ onto in order to guarantee the result is non-zero. We can also extend $\rem_a$ to $B(\infty)$ by considering the smallest weight $\lambda$ such that $b \in B(\infty)$ is non-zero under the projection onto $B(\lambda)$, and then $\rem_a(b) = \inner{\lambda}{h_a}$. We also have following the analog of Proposition~\ref{prop:diff_rem_lambda}.

\begin{prop}
Let $\Psi \colon \RC(\infty) \longrightarrow B(\infty)$ be the canonical isomorphism. Then
\[
\diff_a(\nu, J) = \rem_a(\Psi(\nu, J)).
\]
\end{prop}

We can also interpret $\rem_a$ on (marginally large) tableaux, with no columns of height $n$ in types $B_n$ or $D_{n+1}$, as being the number of basic columns of height $a$ (possibly not full height) that can be removed from a tableaux $T \in \TT(\lambda)$ and sliding the entries of those rows left such that the result is a classical tableaux. Columns of height $n$ in type $B_n$ are counted twice, and in type $D_n$ they contribute to both $\rem_n$ and $\rem_{n+1}$. Additionally, we can extend this interpretation to $T \in \TT(\infty)$.


\bibliography{RC_MLT}{}
\bibliographystyle{amsalpha}
\end{document}

%% file: B2-21.tex
\tiny
\begin{tikzpicture}[>=latex,line join=bevel,yscale=.35,xscale=1]
\node (node_32) at (169.000000bp,104.000000bp) [draw,draw=none] {${\def\lr#1{\multicolumn{1}{|@{\hspace{.6ex}}c@{\hspace{.6ex}}|}{\raisebox{-.3ex}{$#1$}}}\raisebox{-.6ex}{$\begin{array}[b]{*{2}c}\cline{1-2}\lr{0}&\lr{\overline{1}}\\\cline{1-2}\lr{\overline{1}}\\\cline{1-1}\end{array}$}}$};
  \node (node_12) at (154.000000bp,456.000000bp) [draw,draw=none] {${\def\lr#1{\multicolumn{1}{|@{\hspace{.6ex}}c@{\hspace{.6ex}}|}{\raisebox{-.3ex}{$#1$}}}\raisebox{-.6ex}{$\begin{array}[b]{*{2}c}\cline{1-2}\lr{2}&\lr{\overline{2}}\\\cline{1-2}\lr{0}\\\cline{1-1}\end{array}$}}$};
  \node (node_26) at (169.000000bp,192.000000bp) [draw,draw=none] {${\def\lr#1{\multicolumn{1}{|@{\hspace{.6ex}}c@{\hspace{.6ex}}|}{\raisebox{-.3ex}{$#1$}}}\raisebox{-.6ex}{$\begin{array}[b]{*{2}c}\cline{1-2}\lr{0}&\lr{\overline{1}}\\\cline{1-2}\lr{\overline{2}}\\\cline{1-1}\end{array}$}}$};
  \node (node_27) at (13.000000bp,456.000000bp) [draw,draw=none] {${\def\lr#1{\multicolumn{1}{|@{\hspace{.6ex}}c@{\hspace{.6ex}}|}{\raisebox{-.3ex}{$#1$}}}\raisebox{-.6ex}{$\begin{array}[b]{*{2}c}\cline{1-2}\lr{2}&\lr{2}\\\cline{1-2}\lr{\overline{1}}\\\cline{1-1}\end{array}$}}$};
  \node (node_24) at (242.000000bp,280.000000bp) [draw,draw=none] {${\def\lr#1{\multicolumn{1}{|@{\hspace{.6ex}}c@{\hspace{.6ex}}|}{\raisebox{-.3ex}{$#1$}}}\raisebox{-.6ex}{$\begin{array}[b]{*{2}c}\cline{1-2}\lr{2}&\lr{\overline{1}}\\\cline{1-2}\lr{\overline{2}}\\\cline{1-1}\end{array}$}}$};
  \node (node_25) at (136.000000bp,280.000000bp) [draw,draw=none] {${\def\lr#1{\multicolumn{1}{|@{\hspace{.6ex}}c@{\hspace{.6ex}}|}{\raisebox{-.3ex}{$#1$}}}\raisebox{-.6ex}{$\begin{array}[b]{*{2}c}\cline{1-2}\lr{0}&\lr{\overline{2}}\\\cline{1-2}\lr{\overline{2}}\\\cline{1-1}\end{array}$}}$};
  \node (node_22) at (77.000000bp,456.000000bp) [draw,draw=none] {${\def\lr#1{\multicolumn{1}{|@{\hspace{.6ex}}c@{\hspace{.6ex}}|}{\raisebox{-.3ex}{$#1$}}}\raisebox{-.6ex}{$\begin{array}[b]{*{2}c}\cline{1-2}\lr{2}&\lr{0}\\\cline{1-2}\lr{\overline{2}}\\\cline{1-1}\end{array}$}}$};
  \node (node_23) at (95.000000bp,368.000000bp) [draw,draw=none] {${\def\lr#1{\multicolumn{1}{|@{\hspace{.6ex}}c@{\hspace{.6ex}}|}{\raisebox{-.3ex}{$#1$}}}\raisebox{-.6ex}{$\begin{array}[b]{*{2}c}\cline{1-2}\lr{2}&\lr{\overline{2}}\\\cline{1-2}\lr{\overline{2}}\\\cline{1-1}\end{array}$}}$};
  \node (node_20) at (258.000000bp,368.000000bp) [draw,draw=none] {${\def\lr#1{\multicolumn{1}{|@{\hspace{.6ex}}c@{\hspace{.6ex}}|}{\raisebox{-.3ex}{$#1$}}}\raisebox{-.6ex}{$\begin{array}[b]{*{2}c}\cline{1-2}\lr{1}&\lr{\overline{1}}\\\cline{1-2}\lr{\overline{2}}\\\cline{1-1}\end{array}$}}$};
  \node (node_21) at (50.000000bp,544.000000bp) [draw,draw=none] {${\def\lr#1{\multicolumn{1}{|@{\hspace{.6ex}}c@{\hspace{.6ex}}|}{\raisebox{-.3ex}{$#1$}}}\raisebox{-.6ex}{$\begin{array}[b]{*{2}c}\cline{1-2}\lr{2}&\lr{2}\\\cline{1-2}\lr{\overline{2}}\\\cline{1-1}\end{array}$}}$};
  \node (node_28) at (50.000000bp,368.000000bp) [draw,draw=none] {${\def\lr#1{\multicolumn{1}{|@{\hspace{.6ex}}c@{\hspace{.6ex}}|}{\raisebox{-.3ex}{$#1$}}}\raisebox{-.6ex}{$\begin{array}[b]{*{2}c}\cline{1-2}\lr{2}&\lr{0}\\\cline{1-2}\lr{\overline{1}}\\\cline{1-1}\end{array}$}}$};
  \node (node_29) at (91.000000bp,280.000000bp) [draw,draw=none] {${\def\lr#1{\multicolumn{1}{|@{\hspace{.6ex}}c@{\hspace{.6ex}}|}{\raisebox{-.3ex}{$#1$}}}\raisebox{-.6ex}{$\begin{array}[b]{*{2}c}\cline{1-2}\lr{2}&\lr{\overline{2}}\\\cline{1-2}\lr{\overline{1}}\\\cline{1-1}\end{array}$}}$};
  \node (node_9) at (231.000000bp,456.000000bp) [draw,draw=none] {${\def\lr#1{\multicolumn{1}{|@{\hspace{.6ex}}c@{\hspace{.6ex}}|}{\raisebox{-.3ex}{$#1$}}}\raisebox{-.6ex}{$\begin{array}[b]{*{2}c}\cline{1-2}\lr{1}&\lr{\overline{1}}\\\cline{1-2}\lr{0}\\\cline{1-1}\end{array}$}}$};
  \node (node_8) at (213.000000bp,544.000000bp) [draw,draw=none] {${\def\lr#1{\multicolumn{1}{|@{\hspace{.6ex}}c@{\hspace{.6ex}}|}{\raisebox{-.3ex}{$#1$}}}\raisebox{-.6ex}{$\begin{array}[b]{*{2}c}\cline{1-2}\lr{1}&\lr{\overline{2}}\\\cline{1-2}\lr{0}\\\cline{1-1}\end{array}$}}$};
  \node (node_7) at (129.000000bp,632.000000bp) [draw,draw=none] {${\def\lr#1{\multicolumn{1}{|@{\hspace{.6ex}}c@{\hspace{.6ex}}|}{\raisebox{-.3ex}{$#1$}}}\raisebox{-.6ex}{$\begin{array}[b]{*{2}c}\cline{1-2}\lr{1}&\lr{0}\\\cline{1-2}\lr{0}\\\cline{1-1}\end{array}$}}$};
  \node (node_6) at (139.000000bp,720.000000bp) [draw,draw=none] {${\def\lr#1{\multicolumn{1}{|@{\hspace{.6ex}}c@{\hspace{.6ex}}|}{\raisebox{-.3ex}{$#1$}}}\raisebox{-.6ex}{$\begin{array}[b]{*{2}c}\cline{1-2}\lr{1}&\lr{2}\\\cline{1-2}\lr{0}\\\cline{1-1}\end{array}$}}$};
  \node (node_5) at (139.000000bp,808.000000bp) [draw,draw=none] {${\def\lr#1{\multicolumn{1}{|@{\hspace{.6ex}}c@{\hspace{.6ex}}|}{\raisebox{-.3ex}{$#1$}}}\raisebox{-.6ex}{$\begin{array}[b]{*{2}c}\cline{1-2}\lr{1}&\lr{1}\\\cline{1-2}\lr{0}\\\cline{1-1}\end{array}$}}$};
  \node (node_4) at (258.000000bp,544.000000bp) [draw,draw=none] {${\def\lr#1{\multicolumn{1}{|@{\hspace{.6ex}}c@{\hspace{.6ex}}|}{\raisebox{-.3ex}{$#1$}}}\raisebox{-.6ex}{$\begin{array}[b]{*{2}c}\cline{1-2}\lr{1}&\lr{\overline{1}}\\\cline{1-2}\lr{2}\\\cline{1-1}\end{array}$}}$};
  \node (node_3) at (219.000000bp,632.000000bp) [draw,draw=none] {${\def\lr#1{\multicolumn{1}{|@{\hspace{.6ex}}c@{\hspace{.6ex}}|}{\raisebox{-.3ex}{$#1$}}}\raisebox{-.6ex}{$\begin{array}[b]{*{2}c}\cline{1-2}\lr{1}&\lr{\overline{2}}\\\cline{1-2}\lr{2}\\\cline{1-1}\end{array}$}}$};
  \node (node_2) at (202.000000bp,720.000000bp) [draw,draw=none] {${\def\lr#1{\multicolumn{1}{|@{\hspace{.6ex}}c@{\hspace{.6ex}}|}{\raisebox{-.3ex}{$#1$}}}\raisebox{-.6ex}{$\begin{array}[b]{*{2}c}\cline{1-2}\lr{1}&\lr{0}\\\cline{1-2}\lr{2}\\\cline{1-1}\end{array}$}}$};
  \node (node_1) at (184.000000bp,808.000000bp) [draw,draw=none] {${\def\lr#1{\multicolumn{1}{|@{\hspace{.6ex}}c@{\hspace{.6ex}}|}{\raisebox{-.3ex}{$#1$}}}\raisebox{-.6ex}{$\begin{array}[b]{*{2}c}\cline{1-2}\lr{1}&\lr{2}\\\cline{1-2}\lr{2}\\\cline{1-1}\end{array}$}}$};
  \node (node_0) at (162.000000bp,896.000000bp) [draw,draw=none] {${\def\lr#1{\multicolumn{1}{|@{\hspace{.6ex}}c@{\hspace{.6ex}}|}{\raisebox{-.3ex}{$#1$}}}\raisebox{-.6ex}{$\begin{array}[b]{*{2}c}\cline{1-2}\lr{1}&\lr{1}\\\cline{1-2}\lr{2}\\\cline{1-1}\end{array}$}}$};
  \node (node_19) at (295.000000bp,456.000000bp) [draw,draw=none] {${\def\lr#1{\multicolumn{1}{|@{\hspace{.6ex}}c@{\hspace{.6ex}}|}{\raisebox{-.3ex}{$#1$}}}\raisebox{-.6ex}{$\begin{array}[b]{*{2}c}\cline{1-2}\lr{1}&\lr{\overline{2}}\\\cline{1-2}\lr{\overline{2}}\\\cline{1-1}\end{array}$}}$};
  \node (node_18) at (95.000000bp,544.000000bp) [draw,draw=none] {${\def\lr#1{\multicolumn{1}{|@{\hspace{.6ex}}c@{\hspace{.6ex}}|}{\raisebox{-.3ex}{$#1$}}}\raisebox{-.6ex}{$\begin{array}[b]{*{2}c}\cline{1-2}\lr{1}&\lr{0}\\\cline{1-2}\lr{\overline{2}}\\\cline{1-1}\end{array}$}}$};
  \node (node_17) at (57.000000bp,632.000000bp) [draw,draw=none] {${\def\lr#1{\multicolumn{1}{|@{\hspace{.6ex}}c@{\hspace{.6ex}}|}{\raisebox{-.3ex}{$#1$}}}\raisebox{-.6ex}{$\begin{array}[b]{*{2}c}\cline{1-2}\lr{1}&\lr{2}\\\cline{1-2}\lr{\overline{2}}\\\cline{1-1}\end{array}$}}$};
  \node (node_16) at (80.000000bp,720.000000bp) [draw,draw=none] {${\def\lr#1{\multicolumn{1}{|@{\hspace{.6ex}}c@{\hspace{.6ex}}|}{\raisebox{-.3ex}{$#1$}}}\raisebox{-.6ex}{$\begin{array}[b]{*{2}c}\cline{1-2}\lr{1}&\lr{1}\\\cline{1-2}\lr{\overline{2}}\\\cline{1-1}\end{array}$}}$};
  \node (node_15) at (181.000000bp,280.000000bp) [draw,draw=none] {${\def\lr#1{\multicolumn{1}{|@{\hspace{.6ex}}c@{\hspace{.6ex}}|}{\raisebox{-.3ex}{$#1$}}}\raisebox{-.6ex}{$\begin{array}[b]{*{2}c}\cline{1-2}\lr{0}&\lr{\overline{1}}\\\cline{1-2}\lr{0}\\\cline{1-1}\end{array}$}}$};
  \node (node_14) at (154.000000bp,368.000000bp) [draw,draw=none] {${\def\lr#1{\multicolumn{1}{|@{\hspace{.6ex}}c@{\hspace{.6ex}}|}{\raisebox{-.3ex}{$#1$}}}\raisebox{-.6ex}{$\begin{array}[b]{*{2}c}\cline{1-2}\lr{0}&\lr{\overline{2}}\\\cline{1-2}\lr{0}\\\cline{1-1}\end{array}$}}$};
  \node (node_13) at (213.000000bp,368.000000bp) [draw,draw=none] {${\def\lr#1{\multicolumn{1}{|@{\hspace{.6ex}}c@{\hspace{.6ex}}|}{\raisebox{-.3ex}{$#1$}}}\raisebox{-.6ex}{$\begin{array}[b]{*{2}c}\cline{1-2}\lr{2}&\lr{\overline{1}}\\\cline{1-2}\lr{0}\\\cline{1-1}\end{array}$}}$};
  \node (node_34) at (146.000000bp,16.000000bp) [draw,draw=none] {${\def\lr#1{\multicolumn{1}{|@{\hspace{.6ex}}c@{\hspace{.6ex}}|}{\raisebox{-.3ex}{$#1$}}}\raisebox{-.6ex}{$\begin{array}[b]{*{2}c}\cline{1-2}\lr{\overline{2}}&\lr{\overline{1}}\\\cline{1-2}\lr{\overline{1}}\\\cline{1-1}\end{array}$}}$};
  \node (node_11) at (154.000000bp,544.000000bp) [draw,draw=none] {${\def\lr#1{\multicolumn{1}{|@{\hspace{.6ex}}c@{\hspace{.6ex}}|}{\raisebox{-.3ex}{$#1$}}}\raisebox{-.6ex}{$\begin{array}[b]{*{2}c}\cline{1-2}\lr{2}&\lr{0}\\\cline{1-2}\lr{0}\\\cline{1-1}\end{array}$}}$};
  \node (node_10) at (174.000000bp,632.000000bp) [draw,draw=none] {${\def\lr#1{\multicolumn{1}{|@{\hspace{.6ex}}c@{\hspace{.6ex}}|}{\raisebox{-.3ex}{$#1$}}}\raisebox{-.6ex}{$\begin{array}[b]{*{2}c}\cline{1-2}\lr{2}&\lr{2}\\\cline{1-2}\lr{0}\\\cline{1-1}\end{array}$}}$};
  \node (node_31) at (114.000000bp,192.000000bp) [draw,draw=none] {${\def\lr#1{\multicolumn{1}{|@{\hspace{.6ex}}c@{\hspace{.6ex}}|}{\raisebox{-.3ex}{$#1$}}}\raisebox{-.6ex}{$\begin{array}[b]{*{2}c}\cline{1-2}\lr{0}&\lr{\overline{2}}\\\cline{1-2}\lr{\overline{1}}\\\cline{1-1}\end{array}$}}$};
  \node (node_30) at (223.000000bp,192.000000bp) [draw,draw=none] {${\def\lr#1{\multicolumn{1}{|@{\hspace{.6ex}}c@{\hspace{.6ex}}|}{\raisebox{-.3ex}{$#1$}}}\raisebox{-.6ex}{$\begin{array}[b]{*{2}c}\cline{1-2}\lr{2}&\lr{\overline{1}}\\\cline{1-2}\lr{\overline{1}}\\\cline{1-1}\end{array}$}}$};
  \node (node_33) at (122.000000bp,104.000000bp) [draw,draw=none] {${\def\lr#1{\multicolumn{1}{|@{\hspace{.6ex}}c@{\hspace{.6ex}}|}{\raisebox{-.3ex}{$#1$}}}\raisebox{-.6ex}{$\begin{array}[b]{*{2}c}\cline{1-2}\lr{\overline{2}}&\lr{\overline{2}}\\\cline{1-2}\lr{\overline{1}}\\\cline{1-1}\end{array}$}}$};
  \draw [blue,->] (node_18) ..controls (89.283000bp,515.690000bp) and (85.362000bp,496.950000bp)  .. (node_22);
  \definecolor{strokecol}{rgb}{0.0,0.0,0.0};
  \pgfsetstrokecolor{strokecol}
  \draw (97.000000bp,500.000000bp) node {$1$};
  \draw [blue,->] (node_24) ..controls (235.940000bp,251.560000bp) and (231.750000bp,232.590000bp)  .. (node_30);
  \draw (243.000000bp,236.000000bp) node {$1$};
  \draw [blue,->] (node_16) ..controls (72.662000bp,691.560000bp) and (67.587000bp,672.590000bp)  .. (node_17);
  \draw (80.000000bp,676.000000bp) node {$1$};
  \draw [red,->] (node_1) ..controls (189.720000bp,779.690000bp) and (193.640000bp,760.950000bp)  .. (node_2);
  \draw (203.000000bp,764.000000bp) node {$2$};
  \draw [red,->] (node_31) ..controls (116.540000bp,163.690000bp) and (118.280000bp,144.950000bp)  .. (node_33);
  \draw (127.000000bp,148.000000bp) node {$2$};
  \draw [red,->] (node_28) ..controls (58.970000bp,341.520000bp) and (64.747000bp,326.590000bp)  .. (71.000000bp,314.000000bp) .. controls (72.568000bp,310.840000bp) and (74.335000bp,307.590000bp)  .. (node_29);
  \draw (80.000000bp,324.000000bp) node {$2$};
  \draw [blue,->] (node_19) ..controls (283.140000bp,427.440000bp) and (274.870000bp,408.220000bp)  .. (node_20);
  \draw (288.000000bp,412.000000bp) node {$1$};
  \draw [red,->] (node_10) ..controls (167.620000bp,603.560000bp) and (163.210000bp,584.590000bp)  .. (node_11);
  \draw (174.000000bp,588.000000bp) node {$2$};
  \draw [blue,->] (node_0) ..controls (169.020000bp,867.560000bp) and (173.870000bp,848.590000bp)  .. (node_1);
  \draw (183.000000bp,852.000000bp) node {$1$};
  \draw [blue,->] (node_14) ..controls (162.610000bp,339.560000bp) and (168.570000bp,320.590000bp)  .. (node_15);
  \draw (178.000000bp,324.000000bp) node {$1$};
  \draw [blue,->] (node_21) ..controls (38.142000bp,515.440000bp) and (29.874000bp,496.220000bp)  .. (node_27);
  \draw (44.000000bp,500.000000bp) node {$1$};
  \draw [red,->] (node_22) ..controls (82.717000bp,427.690000bp) and (86.638000bp,408.950000bp)  .. (node_23);
  \draw (97.000000bp,412.000000bp) node {$2$};
  \draw [blue,->] (node_8) ..controls (210.840000bp,517.590000bp) and (210.800000bp,502.670000bp)  .. (214.000000bp,490.000000bp) .. controls (214.780000bp,486.910000bp) and (215.890000bp,483.790000bp)  .. (node_9);
  \draw (223.000000bp,500.000000bp) node {$1$};
  \draw [blue,->] (node_22) ..controls (67.735000bp,434.740000bp) and (65.041000bp,428.140000bp)  .. (63.000000bp,422.000000bp) .. controls (59.950000bp,412.820000bp) and (57.266000bp,402.540000bp)  .. (node_28);
  \draw (72.000000bp,412.000000bp) node {$1$};
  \draw [blue,->] (node_23) ..controls (93.730000bp,339.690000bp) and (92.858000bp,320.950000bp)  .. (node_29);
  \draw (103.000000bp,324.000000bp) node {$1$};
  \draw [red,->] (node_13) ..controls (202.790000bp,339.560000bp) and (195.730000bp,320.590000bp)  .. (node_15);
  \draw (208.000000bp,324.000000bp) node {$2$};
  \draw [blue,->] (node_33) ..controls (126.220000bp,77.696000bp) and (129.180000bp,62.787000bp)  .. (133.000000bp,50.000000bp) .. controls (133.860000bp,47.123000bp) and (134.860000bp,44.146000bp)  .. (node_34);
  \draw (142.000000bp,60.000000bp) node {$1$};
  \draw [blue,->] (node_5) ..controls (139.000000bp,779.690000bp) and (139.000000bp,760.950000bp)  .. (node_6);
  \draw (148.000000bp,764.000000bp) node {$1$};
  \draw [red,->] (node_12) ..controls (154.000000bp,427.690000bp) and (154.000000bp,408.950000bp)  .. (node_14);
  \draw (163.000000bp,412.000000bp) node {$2$};
  \draw [blue,->] (node_7) ..controls (133.070000bp,605.650000bp) and (136.020000bp,590.740000bp)  .. (140.000000bp,578.000000bp) .. controls (140.940000bp,575.010000bp) and (142.040000bp,571.920000bp)  .. (node_11);
  \draw (149.000000bp,588.000000bp) node {$1$};
  \draw [red,->] (node_29) ..controls (91.758000bp,253.510000bp) and (93.187000bp,238.580000bp)  .. (97.000000bp,226.000000bp) .. controls (97.923000bp,222.950000bp) and (99.123000bp,219.850000bp)  .. (node_31);
  \draw (106.000000bp,236.000000bp) node {$2$};
  \draw [red,->] (node_30) ..controls (208.130000bp,165.780000bp) and (199.180000bp,150.890000bp)  .. (191.000000bp,138.000000bp) .. controls (188.960000bp,134.790000bp) and (186.780000bp,131.430000bp)  .. (node_32);
  \draw (211.000000bp,148.000000bp) node {$2$};
  \draw [blue,->] (node_6) ..controls (150.220000bp,691.440000bp) and (158.040000bp,672.220000bp)  .. (node_10);
  \draw (168.000000bp,676.000000bp) node {$1$};
  \draw [red,->] (node_15) ..controls (177.190000bp,251.690000bp) and (174.570000bp,232.950000bp)  .. (node_26);
  \draw (185.000000bp,236.000000bp) node {$2$};
  \draw [red,->] (node_3) ..controls (217.090000bp,603.690000bp) and (215.790000bp,584.950000bp)  .. (node_8);
  \draw (225.000000bp,588.000000bp) node {$2$};
  \draw [red,->] (node_2) ..controls (207.400000bp,691.690000bp) and (211.100000bp,672.950000bp)  .. (node_3);
  \draw (221.000000bp,676.000000bp) node {$2$};
  \draw [red,->] (node_0) ..controls (153.250000bp,874.750000bp) and (150.780000bp,868.150000bp)  .. (149.000000bp,862.000000bp) .. controls (146.350000bp,852.820000bp) and (144.220000bp,842.530000bp)  .. (node_5);
  \draw (158.000000bp,852.000000bp) node {$2$};
  \draw [red,->] (node_21) ..controls (55.221000bp,517.680000bp) and (58.757000bp,502.770000bp)  .. (63.000000bp,490.000000bp) .. controls (63.989000bp,487.020000bp) and (65.131000bp,483.940000bp)  .. (node_22);
  \draw (72.000000bp,500.000000bp) node {$2$};
  \draw [red,->] (node_11) ..controls (154.000000bp,515.690000bp) and (154.000000bp,496.950000bp)  .. (node_12);
  \draw (163.000000bp,500.000000bp) node {$2$};
  \draw [blue,->] (node_3) ..controls (229.840000bp,610.670000bp) and (233.180000bp,604.080000bp)  .. (236.000000bp,598.000000bp) .. controls (240.370000bp,588.580000bp) and (244.840000bp,578.040000bp)  .. (node_4);
  \draw (253.000000bp,588.000000bp) node {$1$};
  \draw [blue,->] (node_9) ..controls (218.630000bp,434.910000bp) and (215.610000bp,428.370000bp)  .. (214.000000bp,422.000000bp) .. controls (211.730000bp,412.990000bp) and (211.090000bp,402.850000bp)  .. (node_13);
  \draw (223.000000bp,412.000000bp) node {$1$};
  \draw [red,->] (node_8) ..controls (229.620000bp,531.960000bp) and (232.900000bp,529.880000bp)  .. (236.000000bp,528.000000bp) .. controls (249.950000bp,519.540000bp) and (256.700000bp,521.770000bp)  .. (268.000000bp,510.000000bp) .. controls (275.830000bp,501.840000bp) and (281.900000bp,490.970000bp)  .. (node_19);
  \draw (290.000000bp,500.000000bp) node {$2$};
  \draw [red,->] (node_5) ..controls (119.920000bp,779.190000bp) and (106.400000bp,759.490000bp)  .. (node_16);
  \draw (123.000000bp,764.000000bp) node {$2$};
  \draw [blue,->] (node_26) ..controls (169.000000bp,163.690000bp) and (169.000000bp,144.950000bp)  .. (node_32);
  \draw (178.000000bp,148.000000bp) node {$1$};
  \draw [blue,->] (node_17) ..controls (54.291000bp,610.550000bp) and (53.531000bp,603.980000bp)  .. (53.000000bp,598.000000bp) .. controls (52.188000bp,588.860000bp) and (51.555000bp,578.820000bp)  .. (node_21);
  \draw (62.000000bp,588.000000bp) node {$1$};
  \draw [blue,->] (node_25) ..controls (141.810000bp,253.500000bp) and (145.870000bp,238.570000bp)  .. (151.000000bp,226.000000bp) .. controls (152.260000bp,222.920000bp) and (153.730000bp,219.760000bp)  .. (node_26);
  \draw (160.000000bp,236.000000bp) node {$1$};
  \draw [red,->] (node_6) ..controls (134.910000bp,698.580000bp) and (133.770000bp,692.010000bp)  .. (133.000000bp,686.000000bp) .. controls (131.830000bp,676.890000bp) and (130.970000bp,666.850000bp)  .. (node_7);
  \draw (142.000000bp,676.000000bp) node {$2$};
  \draw [red,->] (node_27) ..controls (24.858000bp,427.440000bp) and (33.126000bp,408.220000bp)  .. (node_28);
  \draw (44.000000bp,412.000000bp) node {$2$};
  \draw [blue,->] (node_20) ..controls (252.920000bp,339.690000bp) and (249.430000bp,320.950000bp)  .. (node_24);
  \draw (260.000000bp,324.000000bp) node {$1$};
  \draw [red,->] (node_7) ..controls (118.630000bp,610.720000bp) and (115.510000bp,604.130000bp)  .. (113.000000bp,598.000000bp) .. controls (109.190000bp,588.720000bp) and (105.530000bp,578.320000bp)  .. (node_18);
  \draw (122.000000bp,588.000000bp) node {$2$};
  \draw [red,->] (node_9) ..controls (239.610000bp,427.560000bp) and (245.570000bp,408.590000bp)  .. (node_20);
  \draw (255.000000bp,412.000000bp) node {$2$};
  \draw [red,->] (node_32) ..controls (161.660000bp,75.562000bp) and (156.590000bp,56.587000bp)  .. (node_34);
  \draw (168.000000bp,60.000000bp) node {$2$};
  \draw [red,->] (node_14) ..controls (139.980000bp,346.820000bp) and (136.710000bp,340.390000bp)  .. (135.000000bp,334.000000bp) .. controls (132.600000bp,325.020000bp) and (132.250000bp,314.890000bp)  .. (node_25);
  \draw (144.000000bp,324.000000bp) node {$2$};
  \draw [red,->] (node_4) ..controls (249.390000bp,515.560000bp) and (243.430000bp,496.590000bp)  .. (node_9);
  \draw (255.000000bp,500.000000bp) node {$2$};
\end{tikzpicture}

%% file: a3mlt.tex
\[
\begin{tikzpicture}[xscale=.25,yscale=.5, every node/.style={scale=0.6}]
\node (3+2+1+2+1+3+1+3+1+3+1+1) at (774bp,22bp) [draw,draw=none] {$\young(*,333,*)$};
  \node (3+2+1+4+2+1+2+1+1+2+2) at (948bp,22bp) [draw,draw=none] {$\young(22,*,4)$};
  \node (3+2+1+4+2+1+2+1+1) at (479bp,222bp) [draw,draw=none] {$\young(*,*,4)$};
  \node (3+2+1+4+2+1+4+2+1+2+1+1) at (338bp,122bp) [draw,draw=none] {$\young(*,*,44)$};
  \node (3+2+1+2+1+3+1+3+1+1+2) at (861bp,22bp) [draw,draw=none] {$\young(2,33,*)$};
  \node (3+2+1+4+2+1+2+1+3+1+3+1+1) at (34bp,22bp) [draw,draw=none] {$\young(*,33,4)$};
  \node (3+2+1+2+1+1) at (771bp,322bp) [draw,draw=none] {$\young(*,*,*)$};
  \node (3+2+1+2+1+1+2+3) at (1204bp,22bp) [draw,draw=none] {$\young(23,*,*)$};
  \node (3+2+1+2+1+3+1+1+2) at (907bp,122bp) [draw,draw=none] {$\young(2,3,*)$};
  \node (3+2+1+2+1+4+1+1) at (565bp,122bp) [draw,draw=none] {$\young(*,4,*)$};
  \node (3+2+1+2+1+3+1+1) at (771bp,222bp) [draw,draw=none] {$\young(*,3,*)$};
  \node (3+2+1+4+2+1+4+2+1+2+1+3+1+1) at (208bp,22bp) [draw,draw=none] {$\young(*,3,44)$};
  \node (3+2+1+4+2+1+2+1+3+1+1) at (156bp,122bp) [draw,draw=none] {$\young(*,3,4)$};
  \node (3+2+1+4+2+1+2+1+3+1+1+2) at (121bp,22bp) [draw,draw=none] {$\young(2,3,4)$};
  \node (3+2+1+2+1+1+2+2+2) at (1122bp,22bp) [draw,draw=none] {$\young(222,*,*)$};
  \node (3+2+1+4+2+1+2+1+1+2) at (692bp,122bp) [draw,draw=none] {$\young(2,*,4)$};
  \node (3+2+1+4+2+1+2+1+1+3) at (692bp,22bp) [draw,draw=none] {$\young(3,*,4)$};
  \node (3+2+1+4+2+1+2+1+4+1+1) at (464bp,22bp) [draw,draw=none] {$\young(*,4,4)$};
  \node (3+2+1+2+1+3+1+4+1+1) at (540bp,22bp) [draw,draw=none] {$\young(*,34,*)$};
  \node (3+2+1+4+2+1+4+2+1+2+1+1+2) at (382bp,22bp) [draw,draw=none] {$\young(2,*,44)$};
  \node (3+2+1+2+1+3+1+1+3) at (1346bp,22bp) [draw,draw=none] {$\young(3,3,*)$};
  \node (3+2+1+2+1+3+1+1+2+2) at (1035bp,22bp) [draw,draw=none] {$\young(22,3,*)$};
  \node (3+2+1+2+1+4+1+1+2) at (616bp,22bp) [draw,draw=none] {$\young(2,4,*)$};
  \node (3+2+1+4+2+1+4+2+1+4+2+1+2+1+1) at (295bp,22bp) [draw,draw=none] {$\young(*,*,444)$};
  \node (3+2+1+2+1+1+2+2) at (998bp,122bp) [draw,draw=none] {$\young(22,*,*)$};
  \node (3+2+1+2+1+1+4) at (1275bp,22bp) [draw,draw=none] {$\young(4,*,*)$};
  \node (3+2+1+2+1+1+3) at (1247bp,122bp) [draw,draw=none] {$\young(3,*,*)$};
  \node (3+2+1+2+1+1+2) at (910bp,222bp) [draw,draw=none] {$\young(2,*,*)$};
  \node (3+2+1+2+1+3+1+3+1+1) at (771bp,122bp) [draw,draw=none] {$\young(*,33,*)$};
  \draw [red,->] (3+2+1+4+2+1+2+1+1) ..controls (399.64bp,196.92bp) and (262.68bp,155.37bp)  .. (3+2+1+4+2+1+2+1+3+1+1);
  \definecolor{strokecol}{rgb}{0.0,0.0,0.0};
  \pgfsetstrokecolor{strokecol}
  \draw (356bp,172bp) node {$2$};
  \draw [red,->] (3+2+1+4+2+1+4+2+1+2+1+1) ..controls (291.16bp,85.69bp) and (264.58bp,65.649bp)  .. (3+2+1+4+2+1+4+2+1+2+1+3+1+1);
  \draw (294bp,72bp) node {$2$};
  \draw [red,->] (3+2+1+4+2+1+2+1+3+1+1) ..controls (112.12bp,85.756bp) and (87.314bp,65.826bp)  .. (3+2+1+4+2+1+2+1+3+1+3+1+1);
  \draw (115bp,72bp) node {$2$};
  \draw [green,->] (3+2+1+2+1+3+1+3+1+1) ..controls (738.03bp,102.79bp) and (733.98bp,101.21bp)  .. (730bp,100bp) .. controls (663.63bp,79.871bp) and (633.18bp,119.74bp)  .. (575bp,82bp) .. controls (564.41bp,75.132bp) and (556.59bp,63.841bp)  .. (3+2+1+2+1+3+1+4+1+1);
  \draw (584bp,72bp) node {$3$};
  \draw [red,->] (3+2+1+2+1+3+1+3+1+1) ..controls (771.32bp,88.935bp) and (771.57bp,74.696bp)  .. (772bp,62bp) .. controls (772.09bp,59.298bp) and (772.2bp,56.492bp)  .. (3+2+1+2+1+3+1+3+1+3+1+1);
  \draw (781bp,72bp) node {$2$};
  \draw [red,->] (3+2+1+4+2+1+2+1+1+2) ..controls (692bp,86.834bp) and (692bp,69.192bp)  .. (3+2+1+4+2+1+2+1+1+3);
  \draw (701bp,72bp) node {$2$};
  \draw [red,->] (3+2+1+2+1+1+2) ..controls (993.38bp,196.75bp) and (1145.2bp,152.6bp)  .. (3+2+1+2+1+1+3);
  \draw (1118bp,172bp) node {$2$};
  \draw [blue,->] (3+2+1+4+2+1+2+1+1+2) ..controls (724.98bp,102.85bp) and (729.03bp,101.24bp)  .. (733bp,100bp) .. controls (790.35bp,82.045bp) and (811.99bp,106.18bp)  .. (867bp,82bp) .. controls (879.71bp,76.413bp) and (879.92bp,70.374bp)  .. (891bp,62bp) .. controls (896.53bp,57.822bp) and (902.46bp,53.542bp)  .. (3+2+1+4+2+1+2+1+1+2+2);
  \draw (900bp,72bp) node {$1$};
  \draw [blue,->] (3+2+1+2+1+3+1+3+1+1) ..controls (810.84bp,99.845bp) and (822.36bp,91.607bp)  .. (831bp,82bp) .. controls (838.57bp,73.582bp) and (844.76bp,62.906bp)  .. (3+2+1+2+1+3+1+3+1+1+2);
  \draw (854bp,72bp) node {$1$};
  \draw [green,->] (3+2+1+2+1+1) ..controls (703.16bp,298.23bp) and (574.25bp,254.97bp)  .. (3+2+1+4+2+1+2+1+1);
  \draw (661bp,272bp) node {$3$};
  \draw [green,->] (3+2+1+4+2+1+2+1+1) ..controls (432.9bp,188.96bp) and (400.42bp,166.38bp)  .. (3+2+1+4+2+1+4+2+1+2+1+1);
  \draw (431bp,172bp) node {$3$};
  \draw [green,->] (3+2+1+2+1+1+2) ..controls (868.32bp,203.78bp) and (844.29bp,193.07bp)  .. (824bp,182bp) .. controls (809.61bp,174.15bp) and (807.89bp,168.87bp)  .. (793bp,162bp) .. controls (767.72bp,150.34bp) and (758.83bp,154.39bp)  .. (733bp,144bp) .. controls (732.16bp,143.66bp) and (731.31bp,143.31bp)  .. (3+2+1+4+2+1+2+1+1+2);
  \draw (833bp,172bp) node {$3$};
  \draw [red,->] (3+2+1+2+1+3+1+1+2) ..controls (917.39bp,88.537bp) and (918.89bp,73.787bp)  .. (913bp,62bp) .. controls (910.51bp,57.016bp) and (907.04bp,52.504bp)  .. (3+2+1+2+1+3+1+3+1+1+2);
  \draw (927bp,72bp) node {$2$};
  \draw [green,->] (3+2+1+4+2+1+4+2+1+2+1+1) ..controls (322.89bp,86.574bp) and (314.97bp,68.505bp)  .. (3+2+1+4+2+1+4+2+1+4+2+1+2+1+1);
  \draw (330bp,72bp) node {$3$};
  \draw [green,->] (3+2+1+2+1+4+1+1) ..controls (532.01bp,97.618bp) and (522.24bp,89.835bp)  .. (514bp,82bp) .. controls (504.17bp,72.655bp) and (494.32bp,61.5bp)  .. (3+2+1+4+2+1+2+1+4+1+1);
  \draw (523bp,72bp) node {$3$};
  \draw [blue,->] (3+2+1+2+1+3+1+1+2) ..controls (962.58bp,98.594bp) and (997.39bp,84.27bp)  .. (1000bp,82bp) .. controls (1009.1bp,74.038bp) and (1016.6bp,63.075bp)  .. (3+2+1+2+1+3+1+1+2+2);
  \draw (1025bp,72bp) node {$1$};
  \draw [green,->] (3+2+1+4+2+1+2+1+3+1+1) ..controls (174.34bp,86.444bp) and (184.04bp,68.16bp)  .. (3+2+1+4+2+1+4+2+1+2+1+3+1+1);
  \draw (196bp,72bp) node {$3$};
  \draw [blue,->] (3+2+1+2+1+4+1+1) ..controls (589.04bp,94.572bp) and (593.56bp,88.305bp)  .. (597bp,82bp) .. controls (601.81bp,73.189bp) and (605.63bp,62.967bp)  .. (3+2+1+2+1+4+1+1+2);
  \draw (615bp,72bp) node {$1$};
  \draw [blue,->] (3+2+1+2+1+1+2+2) ..controls (1036.8bp,96.677bp) and (1047.6bp,89.308bp)  .. (1057bp,82bp) .. controls (1069.5bp,72.275bp) and (1082.6bp,60.714bp)  .. (3+2+1+2+1+1+2+2+2);
  \draw (1089bp,72bp) node {$1$};
  \draw [red,->] (3+2+1+2+1+1+3) ..controls (1282.4bp,85.925bp) and (1301.8bp,66.774bp)  .. (3+2+1+2+1+3+1+1+3);
  \draw (1315bp,72bp) node {$2$};
  \draw [green,->] (3+2+1+2+1+1+3) ..controls (1256.8bp,86.704bp) and (1261.9bp,68.849bp)  .. (3+2+1+2+1+1+4);
  \draw (1273bp,72bp) node {$3$};
  \draw [blue,->] (3+2+1+4+2+1+2+1+1) ..controls (539.52bp,220.66bp) and (608.58bp,214.74bp)  .. (655bp,182bp) .. controls (665.32bp,174.72bp) and (673.37bp,163.57bp)  .. (3+2+1+4+2+1+2+1+1+2);
  \draw (683bp,172bp) node {$1$};
  \draw [red,->] (3+2+1+2+1+1+2+2) ..controls (1048.4bp,104.27bp) and (1077.4bp,93.62bp)  .. (1102bp,82bp) .. controls (1123.9bp,71.66bp) and (1147.3bp,58.281bp)  .. (3+2+1+2+1+1+2+3);
  \draw (1149bp,72bp) node {$2$};
  \draw [green,->] (3+2+1+2+1+3+1+1+2) ..controls (857.55bp,106.68bp) and (830.23bp,96.367bp)  .. (809bp,82bp) .. controls (798.85bp,75.129bp) and (800.92bp,67.567bp)  .. (790bp,62bp) .. controls (738.43bp,35.721bp) and (714.4bp,60.301bp)  .. (3+2+1+2+1+4+1+1+2);
  \draw (818bp,72bp) node {$3$};
  \draw [blue,->] (3+2+1+2+1+3+1+1) ..controls (810.54bp,203.49bp) and (830.16bp,193.41bp)  .. (846bp,182bp) .. controls (858.51bp,172.99bp) and (870.98bp,161.5bp)  .. (3+2+1+2+1+3+1+1+2);
  \draw (879bp,172bp) node {$1$};
  \draw [green,->] (3+2+1+4+2+1+2+1+1+2) ..controls (644.39bp,109.55bp) and (619.87bp,104.04bp)  .. (598bp,100bp) .. controls (543.67bp,89.97bp) and (526.8bp,101.22bp)  .. (475bp,82bp) .. controls (454.24bp,74.296bp) and (433.23bp,61.488bp)  .. (3+2+1+4+2+1+4+2+1+2+1+1+2);
  \draw (484bp,72bp) node {$3$};
  \draw [red,->] (3+2+1+2+1+4+1+1) ..controls (543.51bp,94.763bp) and (540.01bp,88.446bp)  .. (538bp,82bp) .. controls (535.2bp,73.044bp) and (534.66bp,62.949bp)  .. (3+2+1+2+1+3+1+4+1+1);
  \draw (547bp,72bp) node {$2$};
  \draw [green,->] (3+2+1+2+1+1+2+2) ..controls (980.37bp,86.444bp) and (971.04bp,68.16bp)  .. (3+2+1+4+2+1+2+1+1+2+2);
  \draw (987bp,72bp) node {$3$};
  \draw [red,->] (3+2+1+2+1+1) ..controls (771bp,286.83bp) and (771bp,269.19bp)  .. (3+2+1+2+1+3+1+1);
  \draw (780bp,272bp) node {$2$};
  \draw [red,->] (3+2+1+2+1+3+1+1) ..controls (771bp,186.83bp) and (771bp,169.19bp)  .. (3+2+1+2+1+3+1+3+1+1);
  \draw (780bp,172bp) node {$2$};
  \draw [blue,->] (3+2+1+2+1+1) ..controls (812.44bp,291.78bp) and (850.04bp,265.27bp)  .. (3+2+1+2+1+1+2);
  \draw (862bp,272bp) node {$1$};
  \draw [green,->] (3+2+1+2+1+3+1+1) ..controls (719.19bp,211.77bp) and (672.91bp,200.92bp)  .. (637bp,182bp) .. controls (621.13bp,173.64bp) and (605.49bp,161.44bp)  .. (3+2+1+2+1+4+1+1);
  \draw (646bp,172bp) node {$3$};
  \draw [blue,->] (3+2+1+2+1+1+2) ..controls (941.38bp,186.05bp) and (958.38bp,167.12bp)  .. (3+2+1+2+1+1+2+2);
  \draw (971bp,172bp) node {$1$};
  \draw [blue,->] (3+2+1+4+2+1+2+1+3+1+1) ..controls (141.33bp,94.531bp) and (138.31bp,88.121bp)  .. (136bp,82bp) .. controls (132.58bp,72.95bp) and (129.69bp,62.827bp)  .. (3+2+1+4+2+1+2+1+3+1+1+2);
  \draw (145bp,72bp) node {$1$};
  \draw [blue,->] (3+2+1+4+2+1+4+2+1+2+1+1) ..controls (353.46bp,86.574bp) and (361.57bp,68.505bp)  .. (3+2+1+4+2+1+4+2+1+2+1+1+2);
  \draw (373bp,72bp) node {$1$};
  \draw [blue,->] (3+2+1+2+1+1+3) ..controls (1231.9bp,86.574bp) and (1224bp,68.505bp)  .. (3+2+1+2+1+1+2+3);
  \draw (1239bp,72bp) node {$1$};
\end{tikzpicture}
\]

%% file: B2-21-RC.tex
\begin{tikzpicture}[>=latex,line join=bevel,yscale=.4,xscale=.35, every node/.style={scale=0.35}]
\node (node_14) at (874.000000bp,368.000000bp) [draw,draw=none] {${\begin{array}[t]{r|c|l} \cline{2-2} -1 &\phantom{|}& -1 \\ \cline{2-2}  &\phantom{|}& -1 \\ \cline{2-2} \end{array}} \quad {\begin{array}[t]{r|c|c|c|c|l} \cline{2-5} -2 &\phantom{a}&\phantom{a}&\phantom{a}&\phantom{a}& -2 \\ \cline{2-5} \end{array}}$};
  \node (node_12) at (671.000000bp,544.000000bp) [draw,draw=none] {${\begin{array}[t]{r|c|l} \cline{2-2} -1 &\phantom{|}& -1 \\ \cline{2-2}  &\phantom{|}& -1 \\ \cline{2-2} \end{array}} \quad {\begin{array}[t]{r|c|c|l} \cline{2-3} 2 &\phantom{a}&\phantom{a}& 0 \\ \cline{2-3} \end{array}}$};
  \node (node_26) at (596.000000bp,192.000000bp) [draw,draw=none] {${\begin{array}[t]{r|c|c|l} \cline{2-3} -1 &\phantom{|}&\phantom{|}& -1 \\ \cline{2-3} -1 &\phantom{|}& \multicolumn{2 }{l}{ -1 } \\ \cline{2-2} \end{array}} \quad {\begin{array}[t]{r|c|c|c|c|c|l} \cline{2-6} -2 &\phantom{a}&\phantom{a}&\phantom{a}&\phantom{a}&\phantom{a}& -2 \\ \cline{2-6} \end{array}}$};
  \node (node_27) at (250.000000bp,280.000000bp) [draw,draw=none] {${\begin{array}[t]{r|c|c|c|l} \cline{2-4} -1 &\phantom{|}&\phantom{|}&\phantom{|}& -1 \\ \cline{2-4} \end{array}} \quad {\begin{array}[t]{r|c|c|c|c|l} \cline{2-5} -2 &\phantom{a}&\phantom{a}&\phantom{a}&\phantom{a}& -2 \\ \cline{2-5} \end{array}}$};
  \node (node_24) at (876.000000bp,280.000000bp) [draw,draw=none] {${\begin{array}[t]{r|c|c|l} \cline{2-3} -1 &\phantom{|}&\phantom{|}& -2 \\ \cline{2-3} -1 &\phantom{|}& \multicolumn{2 }{l}{ -1 } \\ \cline{2-2} \end{array}} \quad {\begin{array}[t]{r|c|c|c|c|l} \cline{2-5} 0 &\phantom{a}&\phantom{a}&\phantom{a}&\phantom{a}& 0 \\ \cline{2-5} \end{array}}$};
  \node (node_25) at (676.000000bp,280.000000bp) [draw,draw=none] {${\begin{array}[t]{r|c|c|l} \cline{2-3} -1 &\phantom{|}&\phantom{|}& -1 \\ \cline{2-3} -1 &\phantom{|}& \multicolumn{2 }{l}{ -1 } \\ \cline{2-2} \end{array}} \quad {\begin{array}[t]{r|c|c|c|c|l} \cline{2-5} 0 &\phantom{a}&\phantom{a}&\phantom{a}&\phantom{a}& -1 \\ \cline{2-5} \end{array}}$};
  \node (node_22) at (157.000000bp,544.000000bp) [draw,draw=none] {${\begin{array}[t]{r|c|c|l} \cline{2-3} -1 &\phantom{|}&\phantom{|}& -2 \\ \cline{2-3} \end{array}} \quad {\begin{array}[t]{r|c|c|l} \cline{2-3} 0 &\phantom{a}&\phantom{a}& 0 \\ \cline{2-3} \end{array}}$};
  \node (node_23) at (684.000000bp,368.000000bp) [draw,draw=none] {${\begin{array}[t]{r|c|c|l} \cline{2-3} -2 &\phantom{|}&\phantom{|}& -2 \\ \cline{2-3} -1 &\phantom{|}& \multicolumn{2 }{l}{ -1 } \\ \cline{2-2} \end{array}} \quad {\begin{array}[t]{r|c|c|c|l} \cline{2-4} 1 &\phantom{a}&\phantom{a}&\phantom{a}& 0 \\ \cline{2-4} \end{array}}$};
  \node (node_20) at (468.000000bp,280.000000bp) [draw,draw=none] {${\begin{array}[t]{r|c|c|l} \cline{2-3} 1 &\phantom{|}&\phantom{|}& 1 \\ \cline{2-3} \end{array}} \quad {\begin{array}[t]{r|c|c|c|c|c|l} \cline{2-6} -4 &\phantom{a}&\phantom{a}&\phantom{a}&\phantom{a}&\phantom{a}& -4 \\ \cline{2-6} \end{array}}$};
  \node (node_21) at (89.000000bp,368.000000bp) [draw,draw=none] {${\begin{array}[t]{r|c|c|l} \cline{2-3} 1 &\phantom{|}&\phantom{|}& 0 \\ \cline{2-3} \end{array}} \quad {\begin{array}[t]{r|c|c|c|c|l} \cline{2-5} -2 &\phantom{a}&\phantom{a}&\phantom{a}&\phantom{a}& -2 \\ \cline{2-5} \end{array}}$};
  \node (node_28) at (292.000000bp,368.000000bp) [draw,draw=none] {${\begin{array}[t]{r|c|c|c|l} \cline{2-4} -2 &\phantom{|}&\phantom{|}&\phantom{|}& -2 \\ \cline{2-4} \end{array}} \quad {\begin{array}[t]{r|c|c|c|l} \cline{2-4} -1 &\phantom{a}&\phantom{a}&\phantom{a}& -1 \\ \cline{2-4} \end{array}}$};
  \node (node_29) at (346.000000bp,192.000000bp) [draw,draw=none] {${\begin{array}[t]{r|c|c|c|l} \cline{2-4} 0 &\phantom{|}&\phantom{|}&\phantom{|}& 0 \\ \cline{2-4} \end{array}} \quad {\begin{array}[t]{r|c|c|c|c|c|l} \cline{2-6} -3 &\phantom{a}&\phantom{a}&\phantom{a}&\phantom{a}&\phantom{a}& -3 \\ \cline{2-6} \end{array}}$};
  \node (node_9) at (347.000000bp,626.000000bp) [draw,draw=none] {${\begin{array}[t]{r|c|l} \cline{2-2} 1 &\phantom{|}& 0 \\ \cline{2-2} \end{array}} \quad {\begin{array}[t]{r|c|c|l} \cline{2-3} 0 &\phantom{a}&\phantom{a}& -1 \\ \cline{2-3} \end{array}}$};
  \node (node_8) at (389.000000bp,705.000000bp) [draw,draw=none] {${\begin{array}[t]{r|c|l} \cline{2-2} 0 &\phantom{|}& -1 \\ \cline{2-2} \end{array}} \quad {\begin{array}[t]{r|c|l} \cline{2-2} 1 &\phantom{a}& 0 \\ \cline{2-2} \end{array}}$};
  \node (node_7) at (871.000000bp,456.000000bp) [draw,draw=none] {${\begin{array}[t]{r|c|l} \cline{2-2} 1 &\phantom{|}& 1 \\ \cline{2-2} \end{array}} \quad {\begin{array}[t]{r|c|c|c|c|l} \cline{2-5} -4 &\phantom{a}&\phantom{a}&\phantom{a}&\phantom{a}& -4 \\ \cline{2-5} \end{array}}$};
  \node (node_6) at (832.000000bp,544.000000bp) [draw,draw=none] {${\begin{array}[t]{r|c|l} \cline{2-2} 1 &\phantom{|}& 1 \\ \cline{2-2} \end{array}} \quad {\begin{array}[t]{r|c|c|c|l} \cline{2-4} -2 &\phantom{a}&\phantom{a}&\phantom{a}& -3 \\ \cline{2-4} \end{array}}$};
  \node (node_5) at (494.000000bp,787.000000bp) [draw,draw=none] {${\begin{array}[t]{r|c|l} \cline{2-2} -1 &\phantom{|}& -1 \\ \cline{2-2} \end{array}} \quad{\emptyset}$};
  \node (node_4) at (526.000000bp,705.000000bp) [draw,draw=none] {${\begin{array}[t]{r|c|l} \cline{2-2} 0 &\phantom{|}& 0 \\ \cline{2-2} \end{array}} \quad {\begin{array}[t]{r|c|l} \cline{2-2} 1 &\phantom{a}& -1 \\ \cline{2-2} \end{array}}$};
  \node (node_3) at (671.000000bp,626.000000bp) [draw,draw=none] {${\begin{array}[t]{r|c|l} \cline{2-2} 1 &\phantom{|}& 1 \\ \cline{2-2} \end{array}} \quad {\begin{array}[t]{r|c|c|l} \cline{2-3} 0 &\phantom{a}&\phantom{a}& -2 \\ \cline{2-3} \end{array}}$};
  \node (node_2) at (234.000000bp,705.000000bp) [draw,draw=none] {${\emptyset} \quad {\begin{array}[t]{r|c|c|l} \cline{2-3} -2 &\phantom{a}&\phantom{a}& -2 \\ \cline{2-3} \end{array}}$};
  \node (node_1) at (389.000000bp,787.000000bp) [draw,draw=none] {${\emptyset} \quad {\begin{array}[t]{r|c|l} \cline{2-2} 0 &\phantom{a}& -1 \\ \cline{2-2} \end{array}}$};
  \node (node_0) at (449.000000bp,864.000000bp) [draw,draw=none] {${\emptyset}\quad{\emptyset}$};
  \node (node_19) at (125.000000bp,456.000000bp) [draw,draw=none] {${\begin{array}[t]{r|c|c|l} \cline{2-3} 0 &\phantom{|}&\phantom{|}& -1 \\ \cline{2-3} \end{array}} \quad {\begin{array}[t]{r|c|c|c|l} \cline{2-4} -1 &\phantom{a}&\phantom{a}&\phantom{a}& -1 \\ \cline{2-4} \end{array}}$};
  \node (node_18) at (495.000000bp,456.000000bp) [draw,draw=none] {${\begin{array}[t]{r|c|c|l} \cline{2-3} 0 &\phantom{|}&\phantom{|}& 0 \\ \cline{2-3} \end{array}} \quad {\begin{array}[t]{r|c|c|c|l} \cline{2-4} -1 &\phantom{a}&\phantom{a}&\phantom{a}& -2 \\ \cline{2-4} \end{array}}$};
  \node (node_17) at (495.000000bp,368.000000bp) [draw,draw=none] {${\begin{array}[t]{r|c|c|l} \cline{2-3} 1 &\phantom{|}&\phantom{|}& 1 \\ \cline{2-3} \end{array}} \quad {\begin{array}[t]{r|c|c|c|c|l} \cline{2-5} -2 &\phantom{a}&\phantom{a}&\phantom{a}&\phantom{a}& -3 \\ \cline{2-5} \end{array}}$};
  \node (node_16) at (495.000000bp,544.000000bp) [draw,draw=none] {${\begin{array}[t]{r|c|c|l} \cline{2-3} -1 &\phantom{|}&\phantom{|}& -1 \\ \cline{2-3} \end{array}} \quad {\begin{array}[t]{r|c|c|l} \cline{2-3} 0 &\phantom{a}&\phantom{a}& -1 \\ \cline{2-3} \end{array}}$};
  \node (node_33) at (596.000000bp,104.000000bp) [draw,draw=none] {${\begin{array}[t]{r|c|c|c|l} \cline{2-4} -2 &\phantom{|}&\phantom{|}&\phantom{|}& -2 \\ \cline{2-4} -1 &\phantom{|}& \multicolumn{3 }{l}{ -1 } \\ \cline{2-2} \end{array}} \quad {\begin{array}[t]{r|c|c|c|c|c|l} \cline{2-6} -1 &\phantom{a}&\phantom{a}&\phantom{a}&\phantom{a}&\phantom{a}& -1 \\ \cline{2-6} \end{array}}$};
  \node (node_32) at (826.000000bp,192.000000bp) [draw,draw=none] {${\begin{array}[t]{r|c|c|c|l} \cline{2-4} -3 &\phantom{|}&\phantom{|}&\phantom{|}& -3 \\ \cline{2-4} -1 &\phantom{|}& \multicolumn{3 }{l}{ -1 } \\ \cline{2-2} \end{array}} \quad {\begin{array}[t]{r|c|c|c|c|l} \cline{2-5} 0 &\phantom{a}&\phantom{a}&\phantom{a}&\phantom{a}& 0 \\ \cline{2-5} \end{array}}$};
  \node (node_13) at (686.000000bp,456.000000bp) [draw,draw=none] {${\begin{array}[t]{r|c|l} \cline{2-2} -1 &\phantom{|}& -1 \\ \cline{2-2}  &\phantom{|}& -1 \\ \cline{2-2} \end{array}} \quad {\begin{array}[t]{r|c|c|c|l} \cline{2-4} 0 &\phantom{a}&\phantom{a}&\phantom{a}& -1 \\ \cline{2-4} \end{array}}$};
  \node (node_34) at (500.000000bp,16.000000bp) [draw,draw=none] {${\begin{array}[t]{r|c|c|c|l} \cline{2-4} -1 &\phantom{|}&\phantom{|}&\phantom{|}& -1 \\ \cline{2-4} -1 &\phantom{|}& \multicolumn{3 }{l}{ -1 } \\ \cline{2-2} \end{array}} \quad {\begin{array}[t]{r|c|c|c|c|c|c|l} \cline{2-7} -2 &\phantom{a}&\phantom{a}&\phantom{a}&\phantom{a}&\phantom{a}&\phantom{a}& -2 \\ \cline{2-7} \end{array}}$};
  \node (node_11) at (180.000000bp,626.000000bp) [draw,draw=none] {${\begin{array}[t]{r|c|l} \cline{2-2} 1 &\phantom{|}& -1 \\ \cline{2-2} \end{array}} \quad {\begin{array}[t]{r|c|c|l} \cline{2-3} 0 &\phantom{a}&\phantom{a}& 0 \\ \cline{2-3} \end{array}}$};
  \node (node_10) at (324.000000bp,544.000000bp) [draw,draw=none] {${\begin{array}[t]{r|c|l} \cline{2-2} 1 &\phantom{|}& 0 \\ \cline{2-2} \end{array}} \quad {\begin{array}[t]{r|c|c|c|l} \cline{2-4} -2 &\phantom{a}&\phantom{a}&\phantom{a}& -2 \\ \cline{2-4} \end{array}}$};
  \node (node_31) at (309.000000bp,456.000000bp) [draw,draw=none] {${\begin{array}[t]{r|c|c|c|l} \cline{2-4} -3 &\phantom{|}&\phantom{|}&\phantom{|}& -3 \\ \cline{2-4} \end{array}} \quad {\begin{array}[t]{r|c|c|l} \cline{2-3} 0 &\phantom{a}&\phantom{a}& 0 \\ \cline{2-3} \end{array}}$};
  \node (node_30) at (355.000000bp,104.000000bp) [draw,draw=none] {${\begin{array}[t]{r|c|c|c|l} \cline{2-4} 1 &\phantom{|}&\phantom{|}&\phantom{|}& 1 \\ \cline{2-4} \end{array}} \quad {\begin{array}[t]{r|c|c|c|c|c|c|l} \cline{2-7} -4 &\phantom{a}&\phantom{a}&\phantom{a}&\phantom{a}&\phantom{a}&\phantom{a}& -4 \\ \cline{2-7} \end{array}}$};
  \node (node_15) at (495.000000bp,626.000000bp) [draw,draw=none] {${\begin{array}[t]{r|c|c|l} \cline{2-3} -2 &\phantom{|}&\phantom{|}& -2 \\ \cline{2-3} \end{array}} \quad {\begin{array}[t]{r|c|l} \cline{2-2} 1 &\phantom{a}& 0 \\ \cline{2-2} \end{array}}$};
  \draw [red,->] (node_5) ..controls (504.140000bp,760.640000bp) and (512.650000bp,739.380000bp)  .. (node_4);
  \definecolor{strokecol}{rgb}{0.0,0.0,0.0};
  \pgfsetstrokecolor{strokecol}
  \draw (523.000000bp,746.000000bp) node {$2$};
  \draw [blue,->] (node_2) ..controls (216.610000bp,679.200000bp) and (202.370000bp,658.900000bp)  .. (node_11);
  \draw (221.000000bp,664.000000bp) node {$1$};
  \draw [red,->] (node_6) ..controls (842.160000bp,520.600000bp) and (854.670000bp,493.010000bp)  .. (node_7);
  \draw (863.000000bp,500.000000bp) node {$2$};
  \draw [red,->] (node_1) ..controls (339.020000bp,760.200000bp) and (295.940000bp,737.970000bp)  .. (node_2);
  \draw (339.000000bp,746.000000bp) node {$2$};
  \draw [blue,->] (node_8) ..controls (418.660000bp,682.460000bp) and (453.010000bp,657.500000bp)  .. (node_15);
  \draw (467.000000bp,664.000000bp) node {$1$};
  \draw [blue,->] (node_14) ..controls (874.640000bp,339.690000bp) and (875.070000bp,320.950000bp)  .. (node_24);
  \draw (884.000000bp,324.000000bp) node {$1$};
  \draw [red,->] (node_9) ..controls (340.930000bp,603.870000bp) and (334.100000bp,580.130000bp)  .. (node_10);
  \draw (348.000000bp,588.000000bp) node {$2$};
  \draw [blue,->] (node_6) ..controls (794.750000bp,521.060000bp) and (750.250000bp,494.840000bp)  .. (node_13);
  \draw (781.000000bp,500.000000bp) node {$1$};
  \draw [blue,->] (node_13) ..controls (685.360000bp,427.690000bp) and (684.930000bp,408.950000bp)  .. (node_23);
  \draw (694.000000bp,412.000000bp) node {$1$};
  \draw [red,->] (node_32) ..controls (747.630000bp,161.700000bp) and (686.780000bp,138.940000bp)  .. (node_33);
  \draw (740.000000bp,148.000000bp) node {$2$};
  \draw [red,->] (node_29) ..controls (348.320000bp,168.790000bp) and (351.150000bp,141.770000bp)  .. (node_30);
  \draw (361.000000bp,148.000000bp) node {$2$};
  \draw [blue,->] (node_24) ..controls (859.900000bp,251.320000bp) and (848.590000bp,231.850000bp)  .. (node_32);
  \draw (864.000000bp,236.000000bp) node {$1$};
  \draw [red,->] (node_8) ..controls (377.690000bp,683.270000bp) and (365.320000bp,660.590000bp)  .. (node_9);
  \draw (381.000000bp,664.000000bp) node {$2$};
  \draw [blue,->] (node_20) ..controls (500.480000bp,257.180000bp) and (539.010000bp,231.290000bp)  .. (node_26);
  \draw (553.000000bp,236.000000bp) node {$1$};
  \draw [blue,->] (node_19) ..controls (170.880000bp,431.370000bp) and (231.760000bp,400.020000bp)  .. (node_28);
  \draw (233.000000bp,412.000000bp) node {$1$};
  \draw [red,->] (node_31) ..controls (304.610000bp,432.790000bp) and (299.270000bp,405.770000bp)  .. (node_28);
  \draw (312.000000bp,412.000000bp) node {$2$};
  \draw [blue,->] (node_10) ..controls (250.530000bp,529.250000bp) and (222.020000bp,521.370000bp)  .. (198.000000bp,510.000000bp) .. controls (186.750000bp,504.670000bp) and (161.320000bp,485.520000bp)  .. (node_19);
  \draw (207.000000bp,500.000000bp) node {$1$};
  \draw [red,->] (node_19) ..controls (115.620000bp,432.600000bp) and (104.070000bp,405.010000bp)  .. (node_21);
  \draw (120.000000bp,412.000000bp) node {$2$};
  \draw [blue,->] (node_11) ..controls (173.930000bp,603.870000bp) and (167.100000bp,580.130000bp)  .. (node_22);
  \draw (181.000000bp,588.000000bp) node {$1$};
  \draw [red,->] (node_18) ..controls (495.000000bp,432.790000bp) and (495.000000bp,405.770000bp)  .. (node_17);
  \draw (504.000000bp,412.000000bp) node {$2$};
  \draw [red,->] (node_15) ..controls (495.000000bp,603.990000bp) and (495.000000bp,580.580000bp)  .. (node_16);
  \draw (504.000000bp,588.000000bp) node {$2$};
  \draw [red,->] (node_23) ..controls (681.460000bp,339.690000bp) and (679.720000bp,320.950000bp)  .. (node_25);
  \draw (689.000000bp,324.000000bp) node {$2$};
  \draw [blue,->] (node_26) ..controls (596.000000bp,163.690000bp) and (596.000000bp,144.950000bp)  .. (node_33);
  \draw (605.000000bp,148.000000bp) node {$1$};
  \draw [blue,->] (node_3) ..controls (671.000000bp,605.430000bp) and (671.000000bp,585.850000bp)  .. (node_12);
  \draw (680.000000bp,588.000000bp) node {$1$};
  \draw [red,->] (node_17) ..controls (488.010000bp,344.720000bp) and (479.460000bp,317.520000bp)  .. (node_20);
  \draw (493.000000bp,324.000000bp) node {$2$};
  \draw [red,->] (node_33) ..controls (564.410000bp,74.697000bp) and (541.310000bp,54.005000bp)  .. (node_34);
  \draw (565.000000bp,60.000000bp) node {$2$};
  \draw [red,->] (node_12) ..controls (675.760000bp,515.690000bp) and (679.030000bp,496.950000bp)  .. (node_13);
  \draw (688.000000bp,500.000000bp) node {$2$};
  \draw [blue,->] (node_0) ..controls (459.320000bp,845.800000bp) and (471.920000bp,824.800000bp)  .. (node_5);
  \draw (485.000000bp,828.000000bp) node {$1$};
  \draw [red,->] (node_27) ..controls (275.730000bp,255.950000bp) and (308.710000bp,226.410000bp)  .. (node_29);
  \draw (316.000000bp,236.000000bp) node {$2$};
  \draw [red,->] (node_3) ..controls (716.290000bp,602.500000bp) and (772.040000bp,574.790000bp)  .. (node_6);
  \draw (770.000000bp,588.000000bp) node {$2$};
  \draw [red,->] (node_22) ..controls (148.690000bp,520.660000bp) and (138.490000bp,493.260000bp)  .. (node_19);
  \draw (153.000000bp,500.000000bp) node {$2$};
  \draw [red,->] (node_13) ..controls (749.590000bp,425.910000bp) and (798.340000bp,403.610000bp)  .. (node_14);
  \draw (805.000000bp,412.000000bp) node {$2$};
  \draw [red,->] (node_28) ..controls (281.030000bp,344.530000bp) and (267.460000bp,316.750000bp)  .. (node_27);
  \draw (284.000000bp,324.000000bp) node {$2$};
  \draw [red,->] (node_0) ..controls (435.160000bp,845.700000bp) and (418.120000bp,824.400000bp)  .. (node_1);
  \draw (437.000000bp,828.000000bp) node {$2$};
  \draw [red,->] (node_4) ..controls (567.220000bp,682.110000bp) and (616.080000bp,656.160000bp)  .. (node_3);
  \draw (628.000000bp,664.000000bp) node {$2$};
  \draw [blue,->] (node_30) ..controls (391.990000bp,81.060000bp) and (436.190000bp,54.843000bp)  .. (node_34);
  \draw (450.000000bp,60.000000bp) node {$1$};
  \draw [blue,->] (node_9) ..controls (388.410000bp,602.620000bp) and (439.000000bp,575.270000bp)  .. (node_16);
  \draw (439.000000bp,588.000000bp) node {$1$};
  \draw [red,->] (node_25) ..controls (649.900000bp,250.940000bp) and (631.110000bp,230.750000bp)  .. (node_26);
  \draw (652.000000bp,236.000000bp) node {$2$};
  \draw [blue,->] (node_7) ..controls (871.720000bp,434.270000bp) and (872.520000bp,411.460000bp)  .. (node_14);
  \draw (881.000000bp,412.000000bp) node {$1$};
  \draw [blue,->] (node_21) ..controls (133.110000bp,343.440000bp) and (191.420000bp,312.290000bp)  .. (node_27);
  \draw (193.000000bp,324.000000bp) node {$1$};
  \draw [blue,->] (node_1) ..controls (389.000000bp,760.890000bp) and (389.000000bp,740.190000bp)  .. (node_8);
  \draw (398.000000bp,746.000000bp) node {$1$};
  \draw [blue,->] (node_17) ..controls (541.540000bp,344.890000bp) and (597.780000bp,318.170000bp)  .. (node_25);
  \draw (610.000000bp,324.000000bp) node {$1$};
  \draw [blue,->] (node_22) ..controls (198.530000bp,519.500000bp) and (253.230000bp,488.560000bp)  .. (node_31);
  \draw (256.000000bp,500.000000bp) node {$1$};
  \draw [red,->] (node_16) ..controls (495.000000bp,520.790000bp) and (495.000000bp,493.770000bp)  .. (node_18);
  \draw (504.000000bp,500.000000bp) node {$2$};
\end{tikzpicture}

%% file: RC_to_MLT.bbl
\newcommand{\etalchar}[1]{$^{#1}$}
\providecommand{\bysame}{\leavevmode\hbox to3em{\hrulefill}\thinspace}
\providecommand{\MR}{\relax\ifhmode\unskip\space\fi MR }
\providecommand{\MRhref}[2]{%
  \href{http://www.ams.org/mathscinet-getitem?mr=#1}{#2}
}
\providecommand{\href}[2]{#2}
\begin{thebibliography}{KKM{\etalchar{+}}92}

\bibitem[Cha01]{Chari01}
Vyjayanthi Chari, \emph{On the fermionic formula and the
  {K}irillov-{R}eshetikhin conjecture}, Internat. Math. Res. Notices (2001),
  no.~12, 629--654. \MR{1836791 (2002i:17019)}

\bibitem[Cli98]{Cliff98}
Gerald Cliff, \emph{Crystal bases and {Y}oung tableaux}, J. Algebra
  \textbf{202} (1998), no.~1, 10--35. \MR{1614241 (99k:17025)}

\bibitem[DS06]{DS06}
Lipika Deka and Anne Schilling, \emph{New fermionic formula for unrestricted
  {K}ostka polynomials}, J. Combin. Theory Ser. A \textbf{113} (2006), no.~7,
  1435--1461. \MR{2259070 (2008g:05219)}

\bibitem[FOS09]{FOS09}
Ghislain Fourier, Masato Okado, and Anne Schilling,
  \emph{Kirillov-{R}eshetikhin crystals for nonexceptional types}, Adv. Math.
  \textbf{222} (2009), no.~3, 1080--1116. \MR{2553378 (2010j:17028)}

\bibitem[Her06]{Hernandez06}
David Hernandez, \emph{The {K}irillov-{R}eshetikhin conjecture and solutions of
  {$T$}-systems}, J. Reine Angew. Math. \textbf{596} (2006), 63--87.
  \MR{2254805 (2007j:17020)}

\bibitem[Her10]{Hernandez10}
\bysame, \emph{Kirillov-{R}eshetikhin conjecture: the general case}, Int. Math.
  Res. Not. IMRN (2010), no.~1, 149--193. \MR{2576287 (2011c:17029)}

\bibitem[HK02]{HK02}
Jin Hong and Seok-Jin Kang, \emph{Introduction to quantum groups and crystal
  bases}, Graduate Studies in Mathematics, vol.~42, American Mathematical
  Society, Providence, RI, 2002. \MR{1881971 (2002m:17012)}

\bibitem[HKO{\etalchar{+}}99]{HKOTY99}
Goro Hatayama, Atsuo Kuniba, Masato Okado, Taichiro Takagi, and Yasuhiko
  Yamada, \emph{Remarks on fermionic formula}, Recent developments in quantum
  affine algebras and related topics ({R}aleigh, {NC}, 1998), Contemp. Math.,
  vol. 248, Amer. Math. Soc., Providence, RI, 1999, pp.~243--291. \MR{1745263
  (2001m:81129)}

\bibitem[HKO{\etalchar{+}}02]{HKOTT02}
Goro Hatayama, Atsuo Kuniba, Masato Okado, Taichiro Takagi, and Zengo Tsuboi,
  \emph{Paths, crystals and fermionic formulae}, Math{P}hys odyssey, 2001,
  Prog. Math. Phys., vol.~23, Birkh\"auser Boston, Boston, MA, 2002,
  pp.~205--272. \MR{1903978 (2003e:17020)}

\bibitem[HL08]{HL08}
Jin Hong and Hyeonmi Lee, \emph{Young tableaux and crystal
  {$\mathcal{B}(\infty)$} for finite simple {L}ie algebras}, J. Algebra
  \textbf{320} (2008), no.~10, 3680--3693. \MR{2457716 (2009j:17008)}

\bibitem[HL12]{HL12}
\bysame, \emph{Young tableaux and crystal {$\mathcal{B}(\infty)$} for the
  exceptional {L}ie algebra types}, J. Combin. Theory Ser. A \textbf{119}
  (2012), no.~2, 397--419. \MR{2860601 (2012i:17012)}

\bibitem[Kas91]{K91}
Masaki Kashiwara, \emph{On crystal bases of the $q$-analogue of universal
  enveloping algebras}, Duke Math. J. \textbf{63} (1991), no.~2, 465--516.
  \MR{1115118 (93b:17045)}

\bibitem[KKM{\etalchar{+}}92]{KKMMNN91}
Seok-Jin Kang, Masaki Kashiwara, Kailash~C. Misra, Tetsuji Miwa, Toshiki
  Nakashima, and Atsushi Nakayashiki, \emph{Affine crystals and vertex models},
  Infinite analysis, {P}art {A}, {B} ({K}yoto, 1991), Adv. Ser. Math. Phys.,
  vol.~16, World Sci. Publ., River Edge, NJ, 1992, pp.~449--484. \MR{1187560
  (94a:17008)}

\bibitem[KKR86]{KKR86}
S.~V. Kerov, A.~N. Kirillov, and N.~Yu. Reshetikhin, \emph{Combinatorics, the
  {B}ethe ansatz and representations of the symmetric group}, Zap. Nauchn. Sem.
  Leningrad. Otdel. Mat. Inst. Steklov. (LOMI) \textbf{155} (1986),
  no.~Differentsialnaya Geometriya, Gruppy Li i Mekh. VIII, 50--64, 193.
  \MR{869576 (88i:82021)}

\bibitem[KM94]{KM94}
Seok-Jin Kang and Kailash~C. Misra, \emph{Crystal bases and tensor product
  decompositions of {$U_q(G_2)$}-modules}, J. Algebra \textbf{163} (1994),
  no.~3, 675--691. \MR{1265857 (95f:17013)}

\bibitem[KMOY07]{KMOY07}
M.~Kashiwara, K.~C. Misra, M.~Okado, and D.~Yamada, \emph{Perfect crystals for
  {$U_q(D^{(3)}_4)$}}, J. Algebra \textbf{317} (2007), no.~1, 392--423.
  \MR{2360156 (2009b:17035)}

\bibitem[KN94]{KN94}
Masaki Kashiwara and Toshiki Nakashima, \emph{Crystal graphs for
  representations of the {$q$}-analogue of classical {L}ie algebras}, J.
  Algebra \textbf{165} (1994), no.~2, 295--345. \MR{1273277 (95c:17025)}

\bibitem[KR86]{KR86}
A.~N. Kirillov and N.~Yu. Reshetikhin, \emph{The {B}ethe ansatz and the
  combinatorics of {Y}oung tableaux}, Zap. Nauchn. Sem. Leningrad. Otdel. Mat.
  Inst. Steklov. (LOMI) \textbf{155} (1986), no.~Differentsialnaya Geometriya,
  Gruppy Li i Mekh. VIII, 65--115, 194. \MR{869577 (88i:82020)}

\bibitem[KSS02]{KSS1999}
Anatol~N. Kirillov, Anne Schilling, and Mark Shimozono, \emph{A bijection
  between {L}ittlewood-{R}ichardson tableaux and rigged configurations},
  Selecta Math. (N.S.) \textbf{8} (2002), no.~1, 67--135. \MR{1890195
  (2003a:05151)}

\bibitem[LS12]{LS12}
Kyu-Hwan Lee and Ben Salisbury, \emph{A combinatorial description of the
  {G}indikin-{K}arpelevich formula in type {$A$}}, J. Combin. Theory Ser. A
  \textbf{119} (2012), no.~5, 1081--1094. \MR{2891384}

\bibitem[LS14]{LS14}
\bysame, \emph{Young tableaux, canonical bases, and the
  {G}indikin-{K}arpelevich formula}, J. Korean Math. Soc. \textbf{51} (2014),
  no.~2, 289--309. \MR{3178585}

\bibitem[Lus93]{Lusztig93}
George Lusztig, \emph{Introduction to quantum groups}, Progress in Mathematics,
  vol. 110, Birkh\"auser Boston, Inc., Boston, MA, 1993. \MR{1227098
  (94m:17016)}

\bibitem[Nak03]{Nakajima03}
Hiraku Nakajima, \emph{{$t$}-analogs of {$q$}-characters of
  {K}irillov-{R}eshetikhin modules of quantum affine algebras}, Represent.
  Theory \textbf{7} (2003), 259--274 (electronic). \MR{1993360 (2004e:17013)}

\bibitem[OS12]{OS12}
Masato Okado and Nobumasa Sano, \emph{K{KR} type bijection for the exceptional
  affine algebra {$E_6^{(1)}$}}, Algebraic groups and quantum groups, Contemp.
  Math., vol. 565, Amer. Math. Soc., Providence, RI, 2012, pp.~227--242.
  \MR{2932429}

\bibitem[OSS03a]{OSS03}
Masato Okado, Anne Schilling, and Mark Shimozono, \emph{A crystal to rigged
  configuration bijection for nonexceptional affine algebras}, Algebraic
  combinatorics and quantum groups, World Sci. Publ., River Edge, NJ, 2003,
  pp.~85--124. \MR{2035131 (2005b:17037)}

\bibitem[OSS03b]{OSS03II}
\bysame, \emph{Virtual crystals and fermionic formulas of type
  {$D^{(2)}_{n+1},A^{(2)}_{2n}$}, and {$C^{(1)}_n$}}, Represent. Theory
  \textbf{7} (2003), 101--163 (electronic). \MR{1973369 (2004f:17023)}

\bibitem[OSS03c]{OSS03III}
\bysame, \emph{Virtual crystals and {K}leber's algorithm}, Comm. Math. Phys.
  \textbf{238} (2003), no.~1-2, 187--209. \MR{1989674 (2004c:17034)}

\bibitem[OSS13]{OSS13}
Masato Okado, Reiho Sakamoto, and Anne Schilling, \emph{Affine crystal
  structure on rigged configurations of type {$D_n^{(1)}$}}, J. Algebraic
  Combin. \textbf{37} (2013), no.~3, 571--599. \MR{3035517}

\bibitem[S{\etalchar{+}}15]{sage}
W.\thinspace{}A. Stein et~al., \emph{{S}age {M}athematics {S}oftware ({V}ersion
  6.6)}, The Sage Development Team, 2015, {\tt http://www.sagemath.org}.

\bibitem[Sak14]{Sakamoto13}
Reiho Sakamoto, \emph{Rigged configurations and {K}ashiwara operators}, SIGMA
  Symmetry Integrability Geom. Methods Appl. \textbf{10} (2014), Paper 028, 88.
  \MR{3210607}

\bibitem[SCc08]{combinat}
The {S}age-{C}ombinat community, \emph{{S}age-{C}ombinat: enhancing {S}age as a
  toolbox for computer exploration in algebraic combinatorics}, 2008, {\tt
  http://combinat.sagemath.org}.

\bibitem[Sch05]{S05}
Anne Schilling, \emph{A bijection between type {$D^{(1)}_n$} crystals and
  rigged configurations}, J. Algebra \textbf{285} (2005), no.~1, 292--334.
  \MR{2119115 (2006i:17025)}

\bibitem[Sch06]{S06}
\bysame, \emph{Crystal structure on rigged configurations}, Int. Math. Res.
  Not. (2006), Art. ID 97376, 27. \MR{2211139 (2007i:17021)}

\bibitem[Scr15]{Scrimshaw15}
Travis Scrimshaw, \emph{A crystal to rigged configuration bijection and the
  filling map for type {$D_4^{(3)}$}}, 2015, arXiv:1505.05910.

\bibitem[SS06]{SS2006}
Anne Schilling and Mark Shimozono, \emph{{$X=M$} for symmetric powers}, J.
  Algebra \textbf{295} (2006), no.~2, 562--610. \MR{2194969 (2007a:17025)}

\bibitem[SS15a]{SS2015}
Ben Salisbury and Travis Scrimshaw, \emph{A rigged configuration model for
  {$B(\infty)$}}, J. Combin. Theory Ser. A \textbf{133} (2015), 29--57.

\bibitem[SS15b]{SS15}
Anne Schilling and Travis Scrimshaw, \emph{Crystal structure on rigged
  configurations and the filling map}, Electron. J. Combin. \textbf{22} (2015),
  no.~1, Research Paper 73, 56.

\bibitem[Yam98]{Yamane98}
Shigenori Yamane, \emph{Perfect crystals of {$U_q(G^{(1)}_2)$}}, J. Algebra
  \textbf{210} (1998), no.~2, 440--486. \MR{1662347 (2000f:17024)}

\end{thebibliography}
